\documentclass[11pt,letterpaper]{amsart}

\usepackage{amsmath}		
\usepackage{amssymb}        
\usepackage{amsthm}		
\usepackage{color}
\usepackage{enumerate} 		
\usepackage{esint} 		
\usepackage[normalem]{ulem}	
\usepackage{url}		

\numberwithin{equation}{section}

\newtheorem{theorem}{Theorem}[section]
\newtheorem{proposition}{Proposition}[section]
\newtheorem{conjecture}{Conjecture}[section]

\newtheorem{lemma}[proposition]{Lemma}
\newtheorem{example}[proposition]{Example}

\newcommand*{\asubset}{\mathrel{\vcenter{\offinterlineskip\hbox{$\sim$}\vskip-.1ex\hbox{$\subset$}\vskip.2ex}}}
\newcommand{\dd}{\, \mathrm{d}}
\DeclareMathOperator{\dist}{dist}       

\newcommand{\eps}{\varepsilon}
\newcommand{\N}{\mathbb{N}}             
\newcommand{\M}{\mathbb{M}}             

\newcommand{\matrizzz}[1]{\left ( \begin{array}{ccc} #1 \end{array}\right )}
\newcommand{\R}{\mathbb{R}}             

\newcommand{\LL}{\mathcal{L}}             
\newcommand{\MM}{\mathcal{M}}             
\newcommand{\HH}{\mathcal{H}}             
\newcommand{\C}{\mathcal{C}}             
\newcommand{\A}{\mathcal{A}}             

\let\O=\Omega

\newcommand{\res}{\mathop{\hbox{\vrule height 7pt width .5pt depth 0pt \vrule height .5pt width 6pt depth 0pt}}\nolimits\,}
\renewcommand{\vec}{\boldsymbol}    
\newcommand{\vecg}{\boldsymbol}
\newcommand{\weakc}{\rightharpoonup}
\newcommand{\weakcs}{\overset{*}{\rightharpoonup}}

\begin{document}

\title[Reduced models for linearly elastic thin films]{Reduced models for linearly elastic thin films allowing for fracture, debonding or delamination}

\author{Jean-Fran\c cois Babadjian, Duvan Henao}

\address[J.-F. Babadjian]{Sorbonne Universit\'es, UPMC Univ Paris 06, CNRS, UMR 7598, Laboratoire Jacques-Louis Lions, F-75005, Paris, France}
\email{jean-francois.babadjian@upmc.fr}

\address[D. Henao]{Facultad de Matem\'aticas, Pontificia Universidad Cat\'olica de Chile, Vicu\~na Mackenna 4860, Macul, Santiago, Chile}
\email{dhenao@mat.puc.cl}

\subjclass{}
\keywords{Fracture mechanics, Functions of bounded deformation, $\Gamma$-convergence}

\begin{abstract}
This work is devoted so show the appearance
 of different cracking modes in linearly elastic thin film systems by means of an asymptotic analysis as the thickness tends to zero. By superposing two thin plates, and upon suitable scaling law assumptions on the elasticity and fracture parameters, it is proven that either  debonding or transverse cracks can emerge in the limit. A model coupling debonding, transverse cracks and delamination is also discussed.
\end{abstract}

\maketitle

\tableofcontents

\section{Introduction}

\noindent
 It is experimentaly observed that thin films systems can essentially develop two different crack patterns: either transverse cracks channeling through the thickness of the film, or planar debonding at the interface of two layers. In classical fracture mechanics,  a threshold criterion on the energy release rate drives the propagation of a crack along a prescribed path. Within this framework, \cite{HS} described different possibilities of failure modes. In \cite{XH}, a reduced two-dimensional model of a thin film system on an elastic foundation is proposed, and the propagation of different crack modes is discussed. 
This model is later recast as an energy minimization problem, 
based on the variational approach to fracture of \cite{BFM},
first in \cite{LBBMM} in a simplified one-dimensional setting where transverse cracks are represented
by a finite number of discontinuity points for the displacement,
then in \cite{LBBBHM} in the full two-dimensional case.
The total energy is the sum of a bulk energy (including the elastic energy of the film outside the transverse cracks)
 and a surface energy of Griffith type (including the area of the transverse fractures and of the debonded regions). Concretely, it takes the form
 \begin{multline*}
 	E(\bar {\vec u}, \Gamma, \Delta)=  \int_{\omega\setminus \Gamma} 
		\underbrace{\left [\frac{\lambda_f \mu_f}{\lambda_f+2\mu_f} e_{\alpha \alpha} (\bar {\vec u})e_{\beta \beta}(\bar {\vec u}) + \mu_f e_{\alpha \beta}(\bar {\vec u}) e_{\alpha\beta}(\bar {\vec u})
		 \right ]}_{A e(\bar {\vec u}): e(\bar {\vec u})}\dd x
		\\ + \frac{\mu_b}{2}\int_{\omega\setminus \Delta} |\bar {\vec u} - \vec w|^2 \dd x
		+\kappa_f\, \text{length}(\Gamma) + \kappa_b\,\text{area}(\Delta);
 \end{multline*}
 we proceed to explain each term separately.
 The region $\omega\subset \R^2$ denotes the basis of a thin film 
 $\Omega^\eps=\omega \times (0,h_f)$ bonded on a infinitely rigid substrate,
 where $h_f$ is the thickness of the film and $\eps=\frac{h_f}{L}$, $L=\text{diam}\,\omega$,
 is a non-dimensional small parameter. 
 The transverse cracks are of the form $\Gamma\times (0,h_f)$
 where $\Gamma$ is a one-dimensional object which can be thought of 
 as the union of a finite number of closed curves, which are themselves part of the unknowns of the 
 problem.
 The delamination zone $\Delta \subset \omega$ is also an unknown.
 The in-plane displacement of the film
 at the interface with the substrate is denoted by $\bar {\vec u}:\omega \setminus \Gamma 
 \to \R^2$.
 The fracture toughness $\kappa_f$ is a material property of the film,
 while $\kappa_b$ measures the strength of the bonding between
 the film and the substrate.
 The reduced linearly elastic energy $Ae(\bar {\vec u}):e(\bar {\vec u})$ is well-known
 and rigorously derived in the Kirchhoff-Love theory of elastic plates 
 \cite{Ciarlet97}.
 
 It remains to explain the term $E_c:=\frac{\mu_b}{2}\int_{\omega\setminus \Delta} 
 |\bar {\vec u} - \vec w|^2 \dd x$.
 The map $\vec w:\omega \to \R^2$ is given and represents the displacement 
 of the substrate. Since the substrate is  assumed to be infinitely rigid,
 $\vec w$ is the same displacement it would undergo if the film were not present. Note also
 that  we are only considering planar displacements of the substrate. 
 The  energy $E_c$ represents the price to pay in order 
 for the film to deform differently from the substrate.
 It only has to be paid in $\omega\setminus \Delta$ because in $\Delta$
 the film is no longer attached to the substrate.
 By regarding  the film and the substrate as a single elastic body,
it is seen that $E_c$ has the  form $\int_{\omega\setminus \Delta} g([\vec u])\dd x$ of a Barenblatt's cohesive-zone surface energy \cite{Barenblatt},
where $[\vec u]$ represents the jump of the displacement across the debonding
 zone,
in this case with the integrand $g(s)=\frac{\mu_b}{2}s^2$.
These cohesive energies are considered, in particular, in the existing
analytical studies of delamination problems, e.g.\ \cite{BFF,RSZ,FPRZ,FRZ}.
Apart from fracture mechanics, this type of integrals also appear in 
the study of Winkler foundations \cite{Winkler},
with applications as varied as the understanding of the seismic response of piers, 
chromosome function, or the mechanical response of carbon nanotubes embedded in 
elastic media (see \cite{LBB} and the references therein).
In this setting, $E_c$ is interpreted as the effective energy of an elastic foundation, 
understood as a continuous bed of mutually independent, linear, elastic springs,
hence the appearance of a reaction force of the form $\mu_ b (\bar{\vec u}(x)-\vec w(x))$ 
(corresponding to the quadratic energy $\frac{\mu_b}{2} |\bar{\vec u}-\vec w|^2$)
in response to the relative displacement of the body supported on the foundation.

The question addressed by this paper is the rigorous derivation of $E(\bar{\vec u}, \Gamma, \Delta)$ from three-dimensional linearized elasticity in the limit as $\eps\to 0$.
This involves
\begin{enumerate}
\item the derivation of a reduced Griffith model for the initiation and propagation of cracks
in a thin film, and, 
\item the justification of the cohesive energy for the debonding
 at the interface.
\end{enumerate}

Previous studies of the first problem include \cite{BF,BFLM,B1,LBBBHM},
which consider scalar-valued problems or generalizations which 
are incompatible with 3D linear elasticity
(coercivity assumptions
of the form $W(F)\geq C(|F|^p-1)$ are imposed on the stored-energy densities $W(\nabla \vec u)$),
and \cite{FPZ}, where the thin film is linearly elastic but the path and the geometry of the crack
are specified a priori (it has to be of the form $\Gamma\times (0,h_f)$ and it can only be a single crack). 
In Theorem \ref{BH1} below, we present the first 
complete result for the reduction of dimension of a brittle linearly elastic thin film,
without any prior assumption on the geometry or the topology of the cracks. 
As is well-known \cite{BFM}, this falls in  the framework of free discontinuity problems, where a satisfactory mathematical treatment can be done in the space of special functions of bounded deformation. We adopt usual scalings for the elastic and fracture parameters, and show the convergence to a reduced model  where admissible cracks are vertical, and admissible displacements have a Kirchhoff-Love type structure (the out-of-plane displacement is planar, while the in-plane displacement is affine with respect to the out-of-plane variable). The main difficulty is to establish a compactness result on minimizing sequences (Propositions \ref{prop:compactness} and \ref{prop:properties}) showing the structure of limit displacements and cracks with finite energy. It uses tools of geometric measure theory and fine properties of bounded deformation functions. 

The justification of the cohesive energy $E_c$ is also very delicate. 
In \cite{I13,DMI13,FI14,CFI14}
a free discontinuity model 
with a  cohesive fracture energy 
is obtained as the $\Gamma$-limit of an Ambrosio-Tortorelli functional
in which the constraint $z\geq \sqrt{\eps}$ is imposed on the 
internal damage variable, $\eps$ being the width
of the damage zones.
This is in the spirit of considering the possibility that what 
macroscopically would be regarded as fracture is actually
a strain mismatch that is continuously accommodated
through a very thin layer of a very compliant material. 
However, it is unclear whether their approach is suitable for the study of thin films, in particular
if it could explain that only transverse fracture and planar debonding 
at the interface can be observed in the limit. 
Cohesive-type energies have also been obtained
by homogenization
in \cite{Ans,ABZ} 
as the limit of a Neumann sieve, debonding being regarded as
the effect of the interaction of two films through a suitably
periodically distributed contact zone.
A different derivation of a cohesive fracture energy (albeit with a positive activation
threshold) can also be found in \cite{BLZ},
in this case as a result of the homogenization of brittle composites with soft inclusions.

Here we consider the problem
of deriving $E_c$
based on conceiving
the interface between a bimaterial system as
a very thin layer of a third phase,
occupying the region $\Omega_b=\omega \times [-h_b, 0]$,
where $h_b$ represents its thickness.
The expectation is 
to recover the cohesive delamination energy 
from the elastic energy of the bonding layer
when it is made of a material that is increasingly more compliant
as $\eps\to 0$. A scaling law is then proposed for
the Lam\'e moduli of the adhesive, of the form
$(\lambda^\eps, \mu^\eps)=\eps^q(\lambda_b, \mu_b)$,
for some fixed $\lambda_b, \mu_b$ and some exponent $q$ 
to be determined. 
Calling $\eps_b:= \frac{h_b}{L}$ to the aspect ratio in $\Omega_b$
and using the rescaled displacements $u_\alpha(x', x_3):=  
v_\alpha(L x', h_bx_3)$, $u_3(x',x_3)=h_b v_3(Lx', x_3)$,
where $x'\in \omega/L$, $x_3\in [-1,0]$ are non-dimensional
rescaled spatial variables, 
we are able to write the energy of the bonding layer in the form
\begin{align*}
	J_\eps(\vec v_\eps, \Omega_b)=
	h_b \eps^q \eps_b^{-2} \tilde{J}_\eps (\vec v_\eps),
\end{align*}
with
\begin{align*}
	\tilde{J}_\eps (\vec v)
&:= \frac12  \int_{\frac{\omega}{L}\times (-1,0)} \Big\{
\varepsilon_b^2\Big[    e_{\alpha \alpha}(\vec u) e_{\beta \beta}(\vec u) + 2\mu_b  e_{\alpha \beta}(\vec u)  e_{\alpha \beta}(\vec u)\Big]
\\ 
&+\Big[2\lambda_b  e_{\alpha \alpha}(\vec u)  e_{33}(\vec u) + 4\mu_b  e_{\alpha 3}(\vec u)  e_{\alpha 3}(\vec u)\Big] 
+\frac{1}{\varepsilon_b^2}(\lambda_b+2\mu_b)  e_{33}(\vec u)  e_{33}(\vec u)\Big\}\dd x.
\end{align*}
If $\tilde J_\eps$ remains bounded as $\eps\to 0$, 
due to the $\eps_b^{-2}$ in front of the third term, $u_3$ is expected to be planar
in the limit; if the displacement of the substrate is planar, this means $u_3\equiv 0$ 
outside the delamination zone.
 On the other hand, due to the $\eps_b^2$ coefficient for the first term,
we expect the in-plane gradient to be irrelevant.
Thus, the bonding layer is expected to behave 
according to $\frac{\mu_b}{2} \int_{\omega\setminus \Delta} \partial_\alpha u_3\partial_\alpha u_3$, which is minimized if the strain mismatch between the film and 
the substrate is accomodated by an affine transition in the $x_3$ variable,
giving rise to the cohesive energy $E_c$. 
The assumption that $\tilde J_\eps$ is bounded in the asymptotic analysis 
corresponds to the energy in the bonding layer being of the same order of magnitude
as the elastic energy of the film, which scales as $h_f$. This yields the scaling
$$h_b \eps^q \eps_b^{-2} \sim h_f
	\quad \Leftrightarrow\quad 
\frac{h_b}{h_f} \eps^q \sim \left ( \eps \frac{h_b}{h_f}\right )^2 
	\quad \Leftrightarrow\quad 
\frac{h_b}{h_f} \sim \eps^{q-2}.
$$
Without loss of generality, in this paper we consider the case when 
the thicknesses of the film and of the bonding layer have the same order
of magnitude and $q=2$, that is, 
$(\lambda^\eps, \mu^\eps)=\eps^2(\lambda_b, \mu_b)$.
For the effect of other scaling assumptions, we refer to \cite{LBB}.

The above heuristics were made rigorous in \cite{LBBBHM} in the simplified case
of scalar displacements. 
In the anti-plane case, where the problem becomes scalar, a simple adaptation of that result shows the convergence to a model coupling transverse cracks, cohesive transitions as long as the in-plane displacement is below a precise threshold, and delamination when the threshold is overpassed (Theorem \ref{BH}).
In the full vectorial linearly elastic case, 
the reduced model was rigorously derived in \cite{LBB} for the problem of Winkler foundations,
that is, when the displacements are Sobolev maps so that neither the film nor
the bonding layer are allowed to undergo fracture. In Theorem \ref{BH0}
below
we give a simpler proof of the same result. 

We have been unable to prove the convergence to $E(\vec u, \Gamma, \Delta)$
in the case of interest of a linearly elastic bonding layer which may undergo fracture.
We limit ourselves to present some partial results which, in our opinion, 
ought to be considered in any attempt to establish the desired $\Gamma$-convergence.
We prove an energy upper bound by constructing, for every admissible limit displacement, an optimal recovery  sequence (Proposition \ref{pr:UB}).
What remains open is
 to establish the optimality of the affine transitions in the $x_3$ variable
in order to accommodate the mismatch between the film and the substrate.
Indeed, the ability to break gives the bonding layer the opportunity to
reduce its elastic energy by performing a periodic  sequence of small rotations
(Example \ref{ex:micro}). This implies that the delamination zone cannot be identified just by taking the orthogonal projections of the jump set of the displacement, as is done in the Sobolev and scalar cases. As a possible remedy, we consider instead, ``almost vertical'' projections. We are able to prove a surface energy lower bound (although with a bad multiplicative constant)
 and to show the validity of the desired bulk energy lower bound under the assumption that the minimizing sequence satisfies better a priori estimates than just the energy bound (Lemma \ref{lem:bulkb2}).

We end this Introduction by mentioning \cite{MRT}, where a Griffith energy for the debonding at the interface is obtained as the limit elastic energy of a thin bonding layer in a problem involving a damage internal variable. The techniques of that paper may  prove relevant in the derivation of the reduced model $E(\bar{\vec u}, \Gamma, \Delta)$ discussed in this paper.
\bigskip

The paper is organized as follows: Section \ref{sec:2} is devoted to introduce various notations used throughout this work. In Section \ref{sec:3}, we precisely decribe the model and perform a scaling to make the problem more tractable from a mathematical point of view. Section \ref{sec:4} investigates the asymptotic analysis in the absence of cracks, and evidences the appearance of a debonding type limiting energy (Theorem \ref{BH0}). In Section \ref{sec:5}, we carry out the analysis a linearly elastic thin film, and show the emergence of transverse cracks (Theorem \ref{BH1}).  Finally, Section \ref{sec:6} discusses the interplay between transverse cracks, debonding, and delamination.

\section{Notation and preliminaries}\label{sec:2}

\noindent If $a$ and $b \in \R^n$, we write $a \cdot b=\sum_{i=1}^n a_i b_i$ for the Euclidean scalar product, and we denote the norm by $|a|=\sqrt{a \cdot a}$. The open ball of center $x$ and radius $\varrho$ is denoted by $B_\varrho(x)$. If $x=0$, we simply write $B_\varrho$ instead of $B_\varrho(0)$.

\medskip

We denote by $\M^{m \times n}$ the set of real $m \times n$ matrices, and by $\M^{n \times n}_{\rm sym}$ the set of all real symmetric $n \times n$ matrices. Given two matrices $A$ and $B \in \mathbb M^{m \times n}$, we let $A:B:={\rm tr}(A^T B)$ for the Frobenius scalar product, and $|A|:=\sqrt{{\rm tr}(A^T A)}$ for the associated norm ($A^T$ is the transpose of $A$, and ${\rm tr }(A)$ is its trace). We recall that for any two vectors $a \in \R^m$ and $b \in \R^n$, $a \otimes b \in \M^{m \times n}$ stands for the tensor product, {\it i.e.}, $(a \otimes b)_{ij}=a_i b_j$ for all $1 \leq i \leq m$ and $1 \leq j \leq n$. If $m=n$, then $a \odot b:=\frac12 (a \otimes b + b \otimes a) \in \M^{n \times n}_{\rm sym}$ denotes  the symmetric tensor product.

\medskip

Given an open subset $U$ of $\R^n$ and a finite dimensional Euclidean space $X$. We use standard notations for Lebesgues spaces $L^p(U;X)$ and Sobolev spaces $H^1(U;X)$ or $W^{1,p}(U;X)$. We denote by $\mathcal M(U;X)$ the space of all $X$-valued Radon measures with finite total variation. If the target space $X=\R$, we omit to write it for simplicity. According to the Riesz representation Theorem, it is identified to the topological dual of $\C_0(U;X)$ (the space of all continuous functions $\varphi : U \to X$ such that $\{\varphi \geq \eps\}$ is compact for every $\eps>0$), and a weak* topology is defined according to this duality.  The Lebesgue measure in $\R^n$ is denoted by $\LL^n$, and the $k$-dimensional Hausdorff measure by $\HH^k$. Sometimes, the notation $\#$ will be used instead of $\HH^0$ for the counting measure, and $|\cdot|$ instead of the Lebesgue measure $\LL^n$. In dimension $n$, equality of inclusion of sets up to a $\HH^{n-1}$-negligible set will be respectively denoted by $\cong$ and $\asubset$.

\medskip

Given a function $u \in L^1(U;\R^m)$ with $m \geq 1$. We say that $u$ has an approximate limit at $x \in U$ if there exists $\tilde u(x) \in \R^m$ such that
$$\lim_{\varrhoÊ\to 0} \frac{1}{\varrho^n} \int_{B_\varrho(x)} | u(y) -  \tilde u(x)|\dd y =0.$$
The set $S_u$ where this property fails is called the approximate discontinuity set.

We say that $u$ has one-sided Lebesgue limits $u^\pm(x)\in \R^m$ at $x \in U$ with respect to a direction $\nu_u(x) \in \mathbb S^{n-1}:=\{Ê\zeta \in \R^n : |\zeta|=1\}$ if 
$$\lim_{\varrho \to 0} \frac{1}{\varrho^n} \int_{B^\pm_\varrho(x,\nu_{ u}(x))} | u(y) -  u^\pm(x)|\dd y =0,$$
where $B^\pm_\varrho(x,\nu_{ u}(x)):=\{y \in B_\varrho(x) : \pm \nu_{ u}(x) \cdot (y-x) \geq 0\}$. We will denote by $[u](x):=u^+(x) - u^-(x)$ the jump of $u$ at $x$. The jump set $J_u$ of $u$ is defined as the set of points $x \in U$ such that the one-sided Lebesgue limits  with respect to a direction $\nu_{ u}(x)$ exist, and in addition $u^+(x) \neq u^-(x)$.  Clearly we have $J_u \subset S_u$.

\subsection{Functions of bounded variation}

The space $BV(U;\R^m)$ of functions of bounded variation in $U$ with values in $\R^m$ is made of all functions $u \in L^1(U;\R^m)$ such that the distributional derivative satisfies $Du \in \mathcal M(U;\M^{m \times n})$. The measure $Du$ can be decomposed as
$$Du=\nabla u \LL^n + (u^+-u^-) \otimes \nu_uÊ\HH^{n-1} \res J_u + D^c u,$$
where $\nabla u$ is the Radon-Nikod\'ym derivative of $Du$ with respect to the Lebesgue measure $\LL^n$, which coincides with the approximate gradient of $u$. For any $1 \leq i \leq m$ and $1 \leq j \leq n$, we denote by $\partial_j u_i:=(\nabla u)_{ij}$ the entries of $\nabla u$.  The measure $D^c u$ is the Cantor part of $Du$ which has the property of vanishing on any $\sigma$-finite set with respect to the $(n-1)$-dimensional Hausdorff measure $\HH^{n-1}$. The jump set $J_u$ is a  countably $\HH^{n-1}$-rectifiable Borel set, $\nu_u$ is an approximate unit normal to $J_u$, and $u^\pm(x)$ are the one-sided Lebesgue limits of $u$ at $x \in U$ in the direction $\nu_u(x)$. In addition, we have $\HH^{n-1}(S_u \setminus J_u)=0$.

We say that $u$ is a special function of bounded variation, and we write $u \in SBV(U;\R^m)$, if $D^c u=0$. If further $\nabla uÊ\in L^p(U;\R^{m \times n})$ for some $p>1$, and $\HH^{n-1}(J_u)<\infty$, we write $u \in SBV^p(U;\R^m)$. We refer to \cite{AmFuPa00} for general properties of $BV$-functions.

\subsection{Functions of bounded deformation}

The space $BD(U)$ of functions of bounded deformation is made of all vector fields $\vec u \in L^1(U; \R^n)$ whose distributional symmetric gradient satisfies 
$$E\vec u= \frac{D\vec u + D\vec u^T}{2} \in \mathcal M(U;\M^{n \times n}_{\rm sym}).$$
This measure can be decomposed as 
\begin{equation} \label{eq:distSBD}
E\vec u = e(\vec u)\mathcal L^n + (\vec u^+ - \vec u^-) \odot \nu_{\vec u} \HH^{n-1} \res J_{\vec u} +E^c \vec u.
\end{equation}
In the previous expression, $e(\vec u)$ denotes the absolutely continuous part of $E\vec u$ with respect to $\LL^n$. For any $1 \leq i,j \leq n$, we denote by $e_{ij}(\vec u)=(e(\vec u))_{ij}$ the entries of $e(\vec u)$. The measure $E^c\vec u$ is the Cantor part of $E\vec u$ which has the property to vanish on any $\sigma$-finite set with respect to $\HH^{n-1}$. The jump set $J_{\vec u}$ of $\vec u$ is a countably $\HH^{n-1}$-rectifiable Borel set, $\nu_{\vec u}$ is an approximate unit normal to $J_{\vec u}$, and $\vec u^\pm(x)$ are the one-sided Lebesgue limits of $\vec u$ at $x \in U$ in the direction $\nu_{\vec u}(x)$. 
If $E^c \vec u=0$, we say that $\vec u$ is a special function of bounded deformation and we write $\vec u \in SBD(U)$. We refer to \cite{Temam,MSC,ST,Suquet3,AG,B3,AmCoDM97,BeCoDM98,C2,DalMaso} for general properties of $BD$-functions.

\subsection{General conventions}

In the sequel we will always work in dimensions $1$, $2$ or $3$. Latin indices $i$, $j$, $k$, $l$, ... (except $f$ and $b$) take their values in the set $\{1,2,3\}$ unless otherwise indicated. Greek indices $\alpha$, $\beta$, $\gamma$, ... (except $\varepsilon$) take their values in the set $\{1,2\}$. The repeated index summation convention is systematically used.

\section{Description of the problem}\label{sec:3}

\subsection{In the original configuration}

Let  $\omega$ be a bounded and connected open subset of $\R^2$ with Lipschitz boundary which denotes the basis of a thin domain occupying the open set $\O^\eps:=\omega \times (-2\eps,\eps)$ in its reference configuration. We assume that this domain is made of the union of a film $\Omega^\eps_f := \omega \times (0, \varepsilon)$, a bonding layer $\Omega^{\varepsilon}_b := \omega \times [-\varepsilon, 0]$, and a substrate $\Omega^{\varepsilon}_s := \omega \times (-2\varepsilon, -\varepsilon)$. Let us underline that the set $\O_b^\eps$ is not open. Any kinematically admissible displacement $\vec v :\O^\eps \to \R^3$ is required to satisfy the boundary condition $\vec v=0$ in $\O_s^\eps$. In the sequel we shall denote by $x':=(x_1,x_2)$ the in-plane variable.

The background behavior of this medium in that of an isotropic linearly elastic material whose Lam\'e coefficients are given by
$$(\lambda^{\varepsilon}, \mu^{\varepsilon}) 
= \begin{cases} 
(\lambda_f, \mu_f) & \text{in}\ \Omega^\eps_f,\\
\varepsilon^2 (\lambda_b, \mu_b) & \text{in}\ \Omega^\eps_b.
\end{cases}
$$
The elastic energy associated to a displacement $\vec v \in H^1(\O^\eps;\R^3)$ satisfying $\vec v=0$ $\LL^3$-a.e. in $\O_s^\eps$ is given by
\begin{equation}\label{eq:3Delastic-energy}
\frac12 \int_{\O^\eps}Ê\Big[\lambda^\eps e_{ii}(\vec v) e_{jj}(\vec v) + 2\mu^\eps e_{ij}(\vec v) e_{ij}(\vec v)\Big]\dd x.
\end{equation}

If the body undergoes cracks, according to the variational approach to fracture (see \cite{FM,BFM}), the presence of cracks is penalized by means of a surface energy of Griffith type where the toughness is given by
$$\kappa^{\varepsilon}=
\begin{cases}
\kappa_f & \text{in}\ \Omega^\eps_f, \\
\varepsilon \kappa_b & \text{in}\ \Omega^\eps_b.
\end{cases}$$
In this case, Sobolev spaces cannot describe admissible displacements since they may jump across the cracks. The natural framework is to consider displacements which are special functions of bounded deformation. Identifying the cracks with the jump set of the displacement, denoted by $J_{\vec v}$, the surface energy is given by
\begin{equation}\label{eq:3Dsurface-energy}
\int_{J_{\vec v}Ê\cap \O^\eps}\kappa^\eps \dd \HH^2.
\end{equation}
The total energy is then given by the sum of the bulk energy, given by \eqref{eq:3Delastic-energy}, where $e(\vec v)$ is intended as the absolutely continuous part of the strain with respect to the Lebesgue measure, and the surface energy, given by \eqref{eq:3Dsurface-energy}. It is well defined for any displacements $\vec v \in SBD(\O^\eps)$ satisfying the boundary condition $\vec v=0$ $\LL^3$-a.e. in the substrate $\O_s^\eps$.

\subsection{In the rescaled configuration}

As usual in dimension reduction, we rescale the problem on a fixed domain of unit thickness (see \cite{Ciarlet97}). We denote by $\O:=\O^1$, $\O_f:=\O^1_f$, $\O_b:=\O^1_b$, and $\O_s:=\O^1_s$. For every original displacement $\vec v \in H^1(\O^\eps;\R^3)$ (resp. $\vec v \in SBD(\O^\eps)$) such that $\vec v=0$ $\LL^3$-a.e. in $\O_s^\eps$, we define the rescaled displacement $\vec u$ in the rescaled configuration by
$$\begin{cases}
u_\alpha ( x', x_3) =  v_{\alpha}( x',\varepsilon x_3),\\
u_3( x',  x_3) = \eps v_3( x', \varepsilon x_3),
\end{cases} 
\quad \text{for all }x=(x',x_3) \in \O.
$$
Replacing $\vec v$ by this expression in the energy \eqref{eq:3Delastic-energy}, and dividing the resulting expression by $\eps$ yields the following rescaled elastic energy (see \cite{Ciarlet97})
$$J_\eps(\vec u)=J_\eps(\vec u,\O_f) + J_\eps(\vec u,\O_b),$$
where
\begin{eqnarray}\label{eq:expru2}
J_\varepsilon(\vec u, \Omega_f)  &:= & \frac{1}{2} \int_{\Omega_f} \Big[\lambda_f  e_{\alpha \alpha}(\vec u) e_{\beta \beta}(\vec u) + 2\mu_f  e_{\alpha \beta}(\vec u)  e_{\alpha \beta}(\vec u)\Big]\dd x \\
&&+\frac{1}{2\varepsilon^2} \int_{\Omega_f} \Big[2\lambda_f  e_{\alpha \alpha}(\vec u)  e_{33}(\vec u) + 4\mu_f  e_{\alpha 3}(\vec u)  e_{\alpha 3}(\vec u)\Big]\dd x \nonumber\\
&&+\frac{1}{2\varepsilon^4} \int_{\Omega_f} (\lambda_f+2\mu_f)  e_{33}(\vec u)  e_{33}(\vec u)\dd x,\nonumber
\end{eqnarray}
and
\begin{eqnarray}\label{eq:expru3}
J_{\varepsilon}(\vec u,\Omega_b) &:= & \frac{\varepsilon^2}{2} \int_{\Omega_b} 
\Big[\lambda_b  e_{\alpha \alpha}(\vec u) e_{\beta \beta}(\vec u) + 2\mu_b  e_{\alpha \beta}(\vec u)  e_{\alpha \beta}(\vec u)\Big]\dd x \\
&&+\frac{1}{2} \int_{\Omega_b} \Big[2\lambda_b  e_{\alpha \alpha}(\vec u)  e_{33}(\vec u) + 4\mu_b  e_{\alpha 3}(\vec u)  e_{\alpha 3}(\vec u)\Big] \dd x\nonumber\\
&&+\frac{1}{2\varepsilon^2} \int_{\Omega_b} (\lambda_b+2\mu_b)  e_{33}(\vec u)  e_{33}(\vec u)\dd x.\nonumber
\end{eqnarray}
In the case of cracks, the total energy is obtained by adding the surface energy. In the rescaled configuration, it is given by (see \cite{BF,BFLM,B1,B2})
$$E_\eps(\vec u)=E_\eps(\vec u,\O_f) + E_\eps(\vec u,\O_b),$$
where
$$E_{\varepsilon}(\vec u,\Omega_f) = J_{\varepsilon}(\vec u,\Omega_f)+ \kappa_f \int_{J_{\vec u} \cap \Omega_f } \left |\left ( (  \nu_{\vec u})', \frac{1}{\varepsilon} (\nu_{\vec u})_3 \right ) \right | \dd\HH^2,$$
and
$$E_{\varepsilon}(\vec u,\Omega_b) = J_{\varepsilon}(\vec u,\Omega_b)+ \kappa_b \int_{J_{\vec u} \cap \Omega_b} \left |\left ( \varepsilon ( \nu_{\vec u})',  (\nu_{\vec u})_3 \right ) \right | \dd\HH^2.$$

\section{Debonding of thin films}\label{sec:4}

\noindent In this section, we assume that the body is purely elastic, {\it i.e.}, no cracks are allowed. Through an asymptotic analysis as the thickness $\eps$ tends to zero, we rigorously recover a reduced two-dimensional model of a thin film system as an elastic membrane on an in-plane elastic foundation. A similar model has been derived in \cite[Theorem 2.1]{LBB} by means of a different method. The original three-dimensional energy $J_\eps : L^2(\O;\R^3) \to [0,+\infty]$ is defined by 
$$
J_\eps(\vec u):=\left\{
\begin{array}{ll}
J_\eps(\vec u,\O_f) + J_\eps(\vec u,\O_b) & \text{ if } \vec u \in H^1(\O;\R^3) \text{ and }\vec u=0 \: \LL^3\text{-a.e. in }Ê\O_s,\\
+\infty & \text{ otherwise,}
\end{array}
\right.
$$
while the reduced two dimensional energy $J_0:L^2(\O;\R^3) \to [0,+\infty]$ is given by 
$$
J_0(\vec u):=\left\{\!\!\!
\begin{array}{ll}
\begin{array}{l}
\displaystyle \int_\omega\left [ \frac{\lambda_f \mu_f}{\lambda_f + 2\mu_f} e_{\alpha \alpha}(\bar{\vec u})e_{\beta \beta}(\bar{\vec u}) + \mu_f e_{\alpha\beta}(\bar{\vec u})e_{\alpha\beta}(\bar{\vec u}) \right ] \dd x'\\
\hfill \displaystyle + \frac{\mu_b}{2}Ê\int_\omega |\bar{\vec u}|^2 \dd x'
\end{array}
 & \text{if } 
\left\{ \begin{array}{l}
 \vec u=(\bar{\vec u},0),\\
 \bar{\vec u}Ê\in H^1(\omega;\R^2),
 \end{array}\right.\\
+\infty & \text{otherwise.}
\end{array}
\right.
$$

Our first main result in the following $\Gamma$-convergence type result.

\begin{theorem}\label{BH0}
Let $\vec u \in L^2(\O;\R^3)$, then
\begin{itemize}
\item for any sequence $(\vec u_\eps)_{\eps>0} \subset L^2(\O;\R^3)$ with $\vec u_\eps \to \vec u$ strongly in $L^2(\O_f;\R^3)$, then
$$J_0(\vec u)Ê\leq \liminf_{\eps \to 0}ÊJ_\eps (\vec u_\eps);$$
\item there exists a recovery sequence $( {\vec u}^*_\eps)_{\eps>0} \subset L^2(\O;\R^3)$ such that $ {\vec u}^*_\eps \to \vec u$ strongly in $L^2(\O_f;\R^3)$, and
$$J_0(\vec u)Ê\geq \limsup_{\eps \to 0}ÊJ_\eps ({\vec u}^*_\eps).$$
\end{itemize}
\end{theorem}

\begin{proof}
Although some parts of the proof are already well known (see \cite[Theorem 1.11.2]{Ciarlet97}), it will be convenient for us to reproduce the entire argument.

{\bf Step 1. Compactness.} Let $(\vec u_\eps) \subset L^2(\O;\R^3)$ be such that $\vec u_\eps \to \vec u$ strongly in $L^2(\O_f;\R^3)$. If  $\liminf_\eps J_\eps (\vec u_\eps)=+\infty$, there is nothing to prove. We therefore assume that $\liminf_\eps J_\eps (\vec u_\eps)<\infty$. Up to a subsequence, there is no loss of generality to suppose that 
$$J_\eps (\vec u_\eps)=J_\eps(\vec u_\eps,\O_f) + J_\eps(\vec u_\eps,\O_b)\leq C,$$ 
for some constant $C>0$ independent of $\eps$. The expression \eqref{eq:expru2} of the energy in the film $\O_f$ combined with Korn's inequality implies that $(\vec u_\eps)$ is actually bounded in $H^1(\O_f;\R^3)$, and that $\vec u_\eps \weakc \vec u$ weakly in $H^1(\O_f;\R^3)$ with $\vec u \in H^1(\O_f;\R^3)$. Contrary to the case of a standard linearly elastic plate model (see \cite{Ciarlet97}), we will show that, thanks to the Dirichlet condition in the substrate, the limit displacement $\vec u$ is planar instead of just Kirchhoff-Love type. Indeed, using also the expression of the energy \eqref{eq:expru2}--\eqref{eq:expru3}, the fact that $\vec u_\eps=0$ $\LL^3$-a.e. in $\O_s$, and Poincar\'e's inequality, we get that
$$\int_{\O_f}Ê|(u_\eps)_3|^2\dd x \leq \int_{\O_f \cup \O_b}Ê|e_{33}(\vec u_\eps)|^2\dd x \leq C \eps^2 \to 0,$$
so that $u_3=0$. Thanks again to the bound of the energy in the film  \eqref{eq:expru2}, we have
$$\|e_{\alpha 3} (\vec u_\eps)\|_{L^2(\O_f)} \leq C\eps \to 0,$$
which shows that $e_{\alpha 3}(\vec u)=0$. It thus follows that $\partial_3 u_\alpha=-\partial_\alpha u_3=0$ which implies that $u_\alpha(x',x_3)=\bar u_\alpha(x')$ for $\LL^3$-a.e. $x \in \O_f$, for some $\bar{\vec u} \in H^1(\omega;\R^2)$. We have thus identified the right limit space. 

\bigskip

{\bf Step 2. Lower bound.} We next derive the lower bound. Up to a further subsequence, we may assume that
$$\left \{\begin{array}{l}
\textstyle \eps^{-2} e_{33}(\vec u_\eps) \weakc \zeta_3 \smallskip \\
\textstyle \eps^{-1} e_{\alpha 3}(\vec u_\eps) \weakc \zeta_{\alpha}
\end{array}\right .
\ \text{ weakly in}\ L^2(\O_f),$$
for some functions $\zeta_1$, $\zeta_2$ and $\zeta_3 \in L^2(\O_f)$. Then, by lower semicontinuity of the norm with respect to weak convergence, we get that
\begin{multline*}
\liminf_{\eps \to 0} J_{\eps}(\vec u_\eps, \O_f)\\ 
\geq \frac{1}{2} \int_{\O_f}  \big [\lambda_f (e_{\alpha\alpha}(\bar{\vec u}) + \zeta_3)^2 + 2\mu_f e_{\alpha\beta}(\bar{\vec u})e_{\alpha\beta}(\bar{\vec u})+ 4\mu_f \zeta_\alpha \zeta_\alpha + 2\mu_f \zeta_3\zeta_3\big ] \dd x.
\end{multline*}
Minimizing with respect to $(\zeta_1,\zeta_2,\zeta_3)$, we find that the minimal value is attained when $\zeta_\alpha=0$ and $\zeta_3=-\frac{\lambda_f}{\lambda_f+2\mu_f}e_{\alpha\alpha}(\bar{\vec u})$, and thus
$$\liminf_{\eps \to 0} J_{\eps} (\vec u_\eps, \O_f) \geq \int_{\omega}\left [ \frac{\lambda_f \mu_f}{\lambda_f + 2\mu_f} e_{\alpha \alpha}(\bar{\vec u})e_{\beta \beta}(\bar{\vec u})
+ \mu_f e_{\alpha\beta}(\bar{\vec u})e_{\alpha\beta}(\bar{\vec u}) \right] \dd x.$$
We now examine the contribution of the bonding layer. To this aim, according to \eqref{eq:expru3}, isolating the only term of order $1$ leads to
\begin{multline*}
J_\eps(\vec u_\eps,\O_b) \geq  2 \mu_b \int_{\O_b}Êe_{\alpha 3}(\vec u) e_{\alpha 3}(\vec u)\dd x\\
\geq  \frac{\mu_b}{2} \int_{\omega}Ê \left|\int_{-1}^0 [\partial_3 (u_\eps)_1 + \partial_1 (u_\eps)_3] \dd x_3 \right|^2\dd x' +\frac{\mu_b}{2} \int_{\omega}Ê \left|\int_{-1}^0 [\partial_3 (u_\eps)_2 + \partial_2 (u_\eps)_3] \dd x_3 \right|^2\dd x',
\end{multline*}
thanks to the Cauchy-Schwarz inequality with respect to the $x_3$ variable. Since $\vec u_\eps=0$ $\LL^3$-a.e. in $\O_s$, then
$$\int_{-1}^0 \partial_3 \vec u_\eps(x',x_3)\dd x_3=\vec u_\eps(x',0) \quad \text{ for $\LL^2$-a.e. }x' \in \omega,$$
where $\vec u_\eps(\cdot,0)$ denotes the trace of $\vec u_\eps$ on $\{x_3=0\}$. On the other hand, setting $\bar u_3^\eps=\int_{-1}^0 (u_\eps)_3(\cdot,x_3)\dd x_3 \in H^1(\omega)$, we have 
$$\int_{-1}^0 \partial_\alpha (u_\eps)_3(x',x_3)\dd x_3=\partial_\alpha  \bar u_3^\eps (x')\quad \text{ for $\LL^2$-a.e. }x' \in \omega.$$
Gathering everything, we infer that
\begin{equation}\label{1205}
J_\eps(\vec u_\eps,\O_b)  \geq  \frac{\mu_b}{2} \int_{\omega}Ê | (u_\eps)_1(x',0) + \partial_1  \bar u_3^\eps (x') |^2\dd x' + \frac{\mu_b}{2} \int_{\omega}Ê | (u_\eps)_2(x',0) + \partial_2  \bar u_3^\eps (x') |^2\dd x'.
\end{equation}
According to the trace theorem, and since $\bar u_\alpha$ is independent of $x_3$, we have $(u_\eps)_\alpha(\cdot,0) \to \bar u_\alpha$ strongly in $L^2(\omega)$. On the other hand, the energy in the bonding layer  \eqref{eq:expru3} together with the Cauchy-Schwarz and Poincar\'e inequalities yield
$$\int_\omega |\bar u_3^\eps|^2\dd x' \leq \int_{\O_b}|e_{33}(\vec u_\eps)|^2\dd x \leq C \eps^2 \to 0,$$
while \eqref{1205} shows that the sequence $(\nabla \bar u_3^\eps)$ in bounded in $L^2(\omega;\R^2)$. Consequently, $\nabla \bar u_3^\eps \weakc 0$ weakly in $L^2(\omega;\R^2)$, and combining all the convergences established so far, we deduce that
$$\liminf_{\eps \to 0} J_\eps(\vec u_\eps,\O_b)  \geq  \frac{\mu_b}{2} \int_{\omega}Ê | \bar{\vec u} |^2\dd x',$$
which completes the proof of the lower bound.

\bigskip

{\bf Step 3. Upper bound}. We assume without loss of generality that $\vec u=(\bar{\vec u},0)$ for some $\bar{\vec u}Ê\in H^1(\omega;\R^2)$, otherwise the limit energy is infinite. We now define a recovery sequence $({\vec u}^*_\eps)_{\eps>0}$. For all $\eps>0$, let
$${\vec u}^*_\varepsilon( x', x_3 ) =\left\{
\begin{array}{ll}
\big(\bar{\vec u}(x'), \eps^2 x_3 h_\eps( x') \big) & \text{ if }x \in \O_f,\\
(x_3+1)(\bar{\vec u}(x'),0 ) & \text{ if }x \in \O_b,\\
0 &\text{ if } x \in \O_s,
\end{array}
\right.$$
where $(h_\eps)_{\eps>0}$ is a sequence in $\C_c^\infty(\omega)$ such that
\begin{equation}\label{eq:heps3}
h_\eps \to - \frac{\lambda_f}{\lambda_f+2\mu_f} \vec e_{\alpha \alpha}(\bar{\vec u})  \text{ in } L^2(\omega), \quad \lim_{\eps \to 0} \varepsilon \|\nabla h_{\eps}\|_{L^2(\omega)}= 0.
\end{equation}
Clearly, ${\vec u}^*_{\varepsilon}\in H^1(\O;\R^3)$ and ${\vec u}_\eps^*=0$ $\LL^3$-a.e. in $\O_s$. Using \eqref{eq:expru2} we have that 
\begin{align*}
J_\eps ({\vec u}^*_\eps, \Omega_f)	&= \frac{1}{2}\int_{\Omega_f} \big [\lambda_f  e_{\alpha\alpha}(\bar{\vec u})  e_{\beta \beta}(\bar{\vec u})+ 2\mu_f  e_{\alpha \beta}(\bar{\vec u})  e_{\alpha \beta}(\bar{\vec u})\big ] \dd x\\ &
 \quad + \frac{1}{2\varepsilon^2} \int_{\Omega_f} \left [2\lambda_f  e_{\alpha \alpha}(\bar{\vec u}) \varepsilon^2  h_\eps + \mu_f \eps^4x_3^2|\nabla h_\eps|^2 \right ] \dd x \\ 
 & \quad + \frac{1}{2\varepsilon^4} \int_{\Omega_f} (\lambda_f + 2 \mu_f)\varepsilon^4  |h_\eps|^2 \dd x,
\end{align*}
and according to the convergence properties \eqref{eq:heps3}, we get that
\begin{align*}
\nonumber
\lim_{\eps \to 0}& J_{\varepsilon}({\vec u}^*_\eps, \Omega_f) = \frac{1}{2} \int_{\omega} \big [\lambda_f  e_{\alpha\alpha}(\bar{\vec u})  e_{\beta \beta}(\bar{\vec u})	+ 2\mu_f  e_{\alpha \beta}(\bar{\vec u})  e_{\alpha \beta}(\bar{\vec u})\big ] \dd x'\\ 
&	\ - \frac{1}{2} \int_{\omega} \frac{2\lambda_f^2}{\lambda_f+2\mu_f}  e_{\alpha \alpha}(\bar{\vec u})  e_{\beta\beta}(\bar{\vec u}) \dd x' + \frac{1}{2} \int_{\omega} \frac{\lambda_f^2}{\lambda_f + 2 \mu_f}  e_{\alpha \alpha}(\bar{\vec u})  e_{\beta\beta}(\bar{\vec u}) \dd x' \\ \nonumber 
&= \frac{1}{2} \int_{\omega} \left [ \frac{2\lambda_f\mu_f}{\lambda_f+2\mu_f}  e_{\alpha\alpha}(\bar{\vec u})  e_{\beta \beta}(\bar{\vec u}) 	+ 2\mu_f  e_{\alpha \beta}(\bar{\vec u})  e_{\alpha \beta}(\bar{\vec u})\right ] \dd x' .
\end{align*}
On the other hand, \eqref{eq:expru3} yields
$$J_\eps ({\vec u}^*_\eps, \Omega_b) = \frac{\eps^2}{2}\int_{\Omega_b} (x_3+1)^2\big [\lambda_b  e_{\alpha\alpha}(\bar{\vec u})  e_{\beta \beta}(\bar{\vec u})+ 2\mu_b e_{\alpha \beta}(\bar{\vec u})  e_{\alpha \beta}(\bar{\vec u})\big ] \dd x+ \frac{\mu_b}{2} \int_\omega \bar u_\alpha \bar u_\alpha \dd x',$$
and thus
$$\lim_{\eps \to 0} J_{\varepsilon}({\vec u}^*_\eps, \Omega_f) =\frac{\mu_b}{2} \int_\omega |\bar{\vec u}|^2\dd x',$$
which completes the proof of the upper bound.
\end{proof}

\section{Transverse cracks in thin films}\label{sec:5}

\noindent In this section, we assume that the body can fracture. We first only address the analysis of the film $\O_f$ in order to highlight the appearance of transverse cracks in the reduced model. This property is already known in the framework of nonlinear elasticity where energies depend on the deformation gradient \cite{B1,B2,BFLM,BF}. The difficulty here is to consider a linearly elastic material outside the crack so that the energy depends on the elastic strain. 

\subsection{Compactness}\label{sec:compactness}

From a mathematical point of view, the natural functional setting is to consider displacement fields $\vec u \in SBD(\O_f)$. For technical reasons, we also assume that all the
 deformations take place in a fixed container $K$ which is a compact subset of $\R^3$. Therefore, we assume that any displacement is uniformly bounded by some fixed positive constant $M>0$. 

Throughout this section, we assume that $(\vec u_\eps)_{\eps>0} \subset SBD(\O_f)$ is a sequence of displacements in the film such that $Ê\|\vec u_\eps\|_{L^\infty(\O_f)}Ê\leq M$, and
$$\sup_{\eps>0} E_{\eps}(\vec u_\eps,\O_f)<\infty.$$

We establish that any admissible sequence of displacements with uniformly bounded energy converges to some limit displacement having a Kirchhoff-Love type structure.

\begin{proposition}\label{prop:compactness}
Up to a subsequence, there exists $\vec u\in SBD(\O_f) \cap L^\infty(\O_f;\R^3)$ such that
\begin{enumerate}[i)]
\item $\vec u_\eps \to \vec u$ strongly in $L^2(\O_f;\R^3)$ and $\vec u_\eps \weakcs \vec u$ weakly* in $L^\infty(\O_f;\R^3)$;
\item $ e(\vec u_\eps) \weakc  e(\vec u)$ weakly in $L^2(\O_f; \M_{\rm sym}^{3\times 3})$;
\item $e_{\alpha 3}(\vec u)=e_{33}(\vec u)=0$ $\LL^3$-a.e.~in $\O_f$ and $(\nu_{\vec u})_3=0$ $\mathcal H^2$-a.e.~on $J_{\vec u}\cap \O_f$.
\end{enumerate}
\end{proposition}

\begin{proof}
From the hypotheses and the definition of $E_{\eps}(\cdot,\O_f)$, we have that
$$\|\vec u_\eps\|_{L^\infty(\O_f)} + \| e(\vec u_\eps)\|_{L^2(\O_f)} + \HH^2(J_{\vec u_\eps}\cap \O_f) \leq C,$$
for some constant $C>0$ independent of $\eps$. According to the compactness theorem in $SBD$ \cite[Theorem~1.1]{BeCoDM98}, we deduce the existence of a subsequence (not relabeled) and a function $\vec u\in SBD(\O_f)$ such that $\vec u_\eps\to \vec u$ strongly in $L^2(\O_f;\R^3)$, $\vec u_\eps\weakcs \vec u$ weakly* in $L^\infty(\O_f;\R^3)$, $ e(\vec u_\eps) \weakc  e(\vec u)$ weakly in $L^2(\O_f; \M_{\rm sym}^{3\times 3})$, and
\begin{align} \label{eq:liminfSE}
\HH^2(J_{\vec u}\cap \O_f) &\leq \liminf_{\eps\to 0} \HH^2(J_{\vec u_\eps} \cap \O_f) \nonumber\\
&\leq \liminf_{\eps \to 0} \int_{\O_f \cap J_{\vec u_\eps}} \left |\left ( ( \nu_{\vec u_\eps})', \frac{1}{\eps} (\nu_{\vec u_\eps})_3 \right ) \right | \dd\HH^2.
\end{align}

Using the expression of the energy in the film, we deduce that 
\begin{equation} \label{eq:bound_alpha3}
\|e_{\alpha 3} (\vec u_\eps)\|_{L^2(\O_f)} + \int_{\O_f \cap J_{\vec u_\eps}} |(\nu_{\vec u_\eps})_3| \dd\HH^2 \leq C\eps
\end{equation}
and
$$\|e_{33}(\vec u_\eps)\|_{L^2(\O_f)} \leq C\eps^2$$
for some $C$ independent of $\eps$. Using the lower semicontinuity of the left hand side of both equations with respect to the convergences established for $(\vec u_\eps)_{\eps>0}$ (see \cite[Corollary 1.2]{BeCoDM98})
we conclude that $e_{\alpha 3} (\vec u)=e_{33} (\vec u)=  0$ $\LL^3$-a.e. in $\O_f$, and that $( \nu_{\vec u})_3=0$ $\mathcal H^2$-a.e.\ on $J_{\vec u} \cap \O_f$.
\end{proof}

In the sequel, $\vec u$ denotes a displacement as in the conclusion of Proposition \ref{prop:compactness}. Our next goal is to get a more precise structure of such displacements. Contrary to the case of linear elasticity (see \cite{Ciarlet97}) or linearly elastic-perfectly plastic plates (see \cite{DM}), they in general are not of Kirchhoff-Love type ({\it i.e.} such that $E_{i3}\vec u=0$) since we do not control the full distributional  strain $E\vec u$. In particular, the singular part of the shearing strain $E_{\alpha 3} \vec u$ is given by $\frac{[\vec u]_3 \nu_\alpha}{2}\HH^2 \res J_{\vec u}$ which might not vanish. However, we shall prove below that they have the same structure in the sense that the transverse displacement $u_3$ only depends on the planar variable $x'$, while the in-plane displacement $(u_1,u_2)$ is affine with respect to the transverse variable $x_3$.

\begin{proposition}\label{prop:properties}
Let $\vec u\in SBD(\O_f) \cap L^\infty(\O_f;\R^3)$ be such that $e_{i3}(\vec u)= 0$ $\LL^3$-a.e. in $\O_f$, and $(\nu_{\vec u})_3=0$ $\mathcal H^2$-a.e.\ on $J_{\vec u}\cap \O_f$. Then the following properties hold:
\begin{itemize}
\item the function $u_3$ is independent of $x_3$ and it (is identified to a function which) belongs to $SBV(\omega) \cap L^\infty(\omega)$. In addition, its approximate gradient $\nabla u_3=(\partial_1 u_3,\partial_2 u_3)\in SBD(\omega) \cap L^\infty(\omega;\R^2)$; 
\item for $\LL^3$-a.e. $(x',x_3) \in \O_f$, 
\begin{equation}\label{eq:ualphamed}
u_\alpha( x', x_3)= \bar u_\alpha( x') + \left(\frac{1}{2}-x_3\right)\partial_\alpha u_3( x'),
\end{equation}
where $\bar u_\alpha:= \int_0^1 u_\alpha(\cdot, x_3)\dd x_3$, and $\bar{\vec u}:=(\bar u_1, \bar u_2) \in SBD(\omega) \cap L^\infty(\omega;\R^2)$;
\item $J_{\vec u} \cong (J_{\bar{\vec u}}Ê\cup J_{u_3}Ê\cup J_{\nabla u_3}) \times (0,1)$; 
\end{itemize}
\end{proposition}

\begin{proof}
{\bf Step 1.} First of all, by virtue of \eqref{eq:distSBD}, the distributional derivative of $u_3$ with respect to $x_3$ satisfies
$$D_3u_3=E_{33}\vec u= e_{33}(\vec u) \mathcal L^3 + [\vec u]_3(\nu_{\vec u})_3\mathcal H^2\res J_{\vec u}=0.$$
This implies that $u_3$ is independent of $x_3$, and that it can be identified to a function defined on $\omega$.

\medskip

{\bf Step 2.} We next show that $u_3 \in SBV(\omega)$ and that formula \eqref{eq:ualphamed} holds. This will be obtained thanks to a suitable mollification of $\vec u$. We first extend $\vec u$ to the whole space in the following way: since the trace of an $SBD(\O_f)$ function belongs to $L^1(\partial \O_f;\R^3)$ (see \cite[Theorem 3.2]{B3}), according to Gagliardo's Theorem, $\vec u$ may be extended to $\R^3$ by a function, still denoted by $\vec u$, that is compactly supported in $\R^3$ and such that $\vec u \in W^{1,1}(\R^3\setminus \O_f;\R^3)$ with $|E{\vec u}|(\partial \O_f)=0$.

Let $\chi \in \C^\infty_c(\R)$ be an even and non negative function such that $\int_\R \chi(t)\dd t=1$ and ${\rm Supp}Ê\chi \subset (-1,1)$. For all $x=(x_1,x_2,x_3)=(x',x_3) \in \R^3$, we define $\bar \rho(x'):=\chi(x_1)\chi(x_2)$ and $\rho(x):=\chi(x_1)\chi(x_2)\chi(x_3)$. We then denote by $\bar \rho_\delta(x')=\delta^{-2}\bar \rho( x'/\delta)$  a sequence of two-dimensional mollifiers, and by $\rho_\delta( x)=\delta^{-3}\rho( x/\delta)$  a sequence of three-dimensional mollifiers. Since $\vec u\ast \rho_\delta \in \C^1(\R^3;\R^3)$ and
$$\partial_3(\vec u\ast \rho_\delta)_{\alpha} = 2e_{\alpha 3}(\vec u\ast \rho_\delta) - \partial_{\alpha}(\vec u \ast \rho_\delta)_3,$$
it follows from the fundamental theorem of calculus that for each $(x',x_3) \in \O_f$,
\begin{multline} \label{eq:eA3conv}
u_\alpha \ast \rho_\delta( x', x_3) = u_\alpha \ast \rho_\delta( x', 0)+2\int_0^{x_3} e_{\alpha 3}(\vec u \ast \rho_\delta)( x', s) \dd s \\ 
-\int_0^{x_3} \partial_\alpha ( u_3 \ast \rho_\delta)( x',s) \dd s.
\end{multline}

Let us study each of the above terms separately. The term in the left hand side of \eqref{eq:eA3conv} clearly satisfies $u_\alpha \ast \rho_\deltaÊ\to u_\alpha$ strongly in $L^2(\O_f)$, and thus (for a suitable subsequence)
\begin{equation}\label{eq:term1}
u_\alpha \ast \rho_\deltaÊ\to u_\alpha \quad \text{ $\LL^3$-a.e. in } \O_f.
\end{equation}

Concerning the first term on the right-hand side of \eqref{eq:eA3conv}, standard properties of convolution of measures ensure that $E\vec u_\delta \weakcs E\vec u$ weakly* in $\MM(\R^3)$ and $|E\vec u_\delta|(\R^3) \to |E\vec u|(\R^3)$. Therefore, since $|E\vec u|(\partial \O_f)=0$, we deduce that $|E\vec u_\delta|(\O_f) \to |E\vec u|(\O_f)$ which implies, by the continuity property of the trace (see \cite[Proposition 3.4]{B3}) that $u_\alpha \ast \rho_\delta \to u_\alpha$ strongly in $L^1(\partial \O_f)$. Thus, denoting by $u^+_\alpha(\cdot, 0)$ the upper trace of $u_\alpha$ on $\omega \times \{0\}$, there is a subsequence such that
\begin{equation}\label{eq:term2}
u_\alpha \ast \rho_\delta(\cdot, 0)\to u^+_\alpha(\cdot, 0) \quad \text{ $\LL^2$-a.e. in } \omega.
\end{equation}
	
Regarding the second term on the right-hand side of \eqref{eq:eA3conv}, we have $e_{\alpha 3}(\vec u \ast \rho_\delta)= 
(E_{\alpha 3} \vec u)\ast \rho_\delta$ with $E_{\alpha 3} \vec u= \frac{[\vec u]_3(\nu_{\vec u})_\alpha}{2} \mathcal H^2\res J_{\vec u}$, and thus
\begin{eqnarray*}
\mathcal E(x',x_3) & := & \int_0^{x_3} e_{\alpha 3}(\vec u \ast \rho_\delta)(x',s)\dd s\\
& = & \frac12 \int_0^{x_3} \int_{J_{\vec u}} \rho_\delta(x'-y',s-y_3)[\vec u]_3(y)(\nu_{\vec u})_\alpha(y)\dd \HH^2(y)\dd s.
\end{eqnarray*}
Since $\vec u\in L^\infty(\O_f;\R^3)$ with $\|u\|_{L^\infty(\O_f)}\leq M$, then $|[\vec u]|Ê\leq 2 M$ which leads to
\begin{eqnarray*}
|\mathcal E(x',x_3)| & \leq &M \int_0^1 \int_{J_{\vec u}} \rho_\delta(x'-y',s-y_3)\dd \HH^2(y)\dd s\\
& = & M \int_{J_{\vec u}} \int_0^1 \rho_\delta(x'-y',s-y_3)\dd s \dd \HH^2(y),
\end{eqnarray*}
where we used Fubini's Theorem in the last equality. We next denote by $Q'(x',\delta):=x' + (-\delta,\delta)^2$ the open square of $\R^2$ (parallel to the coordinate axis) centered at $x'$ and of edge length $2\delta$. Observing that $\rho_\delta(x'-y',s-y_3)=0$ if $y' \not\in Q'(x',\delta)$ and that $\rho_\delta(x'-y',s-y_3)=\bar\rho_\delta(x'-y') \delta^{-1}\chi((s-y_3)/\delta)$ with $\int_\R \chi(t)\dd t =1$, we get that
\begin{multline*}
|\mathcal E(x',x_3)| \leq  M \int_{J_{\vec u} \cap [Q'(x',\delta) \times (0,1)]} \bar\rho_\delta(x'-y') \left( \int_\R \delta^{-1}\chi((s-y_3)/\delta)\dd s\right) \dd \HH^2(y)\\
 =    M \int_{J_{\vec u} \cap [Q'(x',\delta) \times (0,1)]} \bar\rho_\delta(x'-y') \dd \HH^2(y).
\end{multline*}
For any Borel set $B \subset \omega$, let us define the measure $\mu(B):=\HH^2(J_{\vec u} \cap (B \times (0,1)))$ which is nothing but the push-forward of $\HH^2\res J_{\vec u}$ by the orthogonal projection $\pi : \R^3 \to \R^2 \times \{0\}$. Note that $\mu$ is concentrated on $\pi(J_{\vec u})$ since $\mu (\omega \setminus \pi(J_{\vec u}))= \HH^2(J_{\vec u} \cap [(\omega \setminus \pi(J_{\vec u})) \times (0,1)])=0$. On the other hand, the generalized coarea formula (see \cite[Theorem 293]{AmFuPa00}) yields
$$\LL^2(\pi(J_{\vec u})) \leq \int_{\pi(J_{\vec u})} \HH^0(J_{\vec u}Ê\cap \pi^{-1}(x'))\, dx' = \int_{J_{\vec u}} |(\nu_{\vec u})_3|\dd \HH^2=0.$$
Therefore, $\mu$ and $\LL^2$ are mutually singular which ensures that the Radon-Nikod\'ym derivative $\frac{d\mu}{d\LL^2}(x')=0$ at $\LL^2$-a.e. $x'\in \omega$. It follows that for $\LL^2$-a.e. $x' \in \omega$,
$$\sup_{x_3 \in (0,1)} |\mathcal E(x',x_3)| \leq M \|\chi\|_{L^\infty(\R)}^2 \frac{\mu(Q'(x',\delta))}{\delta^2}\to 0,$$
and thus, in particular,
\begin{equation}\label{eq:term3}
 \int_0^{x_3} e_{\alpha 3}(\vec u \ast \rho_\delta)(x',s)\dd s \to 0  \text{  for $\LL^3$-a.e. }(x',x_3) \in \O_f.
\end{equation}

For what concerns the last term on the right-hand side of \eqref{eq:eA3conv}, since $u_3$ is independent of $x_3$, we infer that $u_3 \ast \rho_\delta$ is independent of $x_3$ as well since $u_3 \ast \rho_\delta(x)=u_3 \ast \bar\rho_\delta(x')$ for all $x \in \R^3$. Therefore,
$$\int_0^{x_3} \partial_\alpha (u_3 \ast \rho_\delta)( x',s) \dd s= x_3 \partial_\alpha (u_3\ast \bar \rho_\delta)( x'),$$
and \eqref{eq:eA3conv} -- \eqref{eq:term3} thus imply that 
$$ \partial_\alpha ( u_3 \ast \bar \rho_\delta)( x')\to \frac{u_\alpha^+( x', 0) - u_\alpha( x', x_3)}{x_3}:= \psi_\alpha( x')\quad \text{for $\LL^3$-a.e. $(x',x_3) \in \O_f$}.$$	
That $\psi_\alpha$ only depends on $x'$ is due to the fact that the left-hand side only depends on $x'$. Moreover, since $u^+_\alpha(\cdot,0) \in L^1(\omega)$ and $u_\alpha(\cdot,x_3) \in L^2(\omega)$ for a.e. $x_3 \in (0,1)$, we deduce that $\psi_\alpha \in L^1 (\omega)$. From the last formula we get that 
\begin{equation}\label{eq:ualpha}
u_{\alpha}(x', x_3) = u^+_\alpha(x', 0) - x_3 \psi_\alpha(x'),
\end{equation}
which in particular implies that 
$D_3 u_\alpha = -\psi_\alpha \mathcal L^3$, and
$$D_\alpha u_3= - D_3 u_\alpha + 2E_{\alpha 3}\vec u= \psi_\alpha \mathcal L^3 + [\vec u]_3 (\nu_{\vec u})_\alpha \mathcal H^2 \res J_{\vec u}.$$
As a consequence, the distributional derivative in $\O_f$ of $u_3$ is a bounded Radon measure in $\O_f$, and therefore $u_3 \in BV(\O_f)$. Since the singular part of  the above measure is concentrated on $J_{\vec u}$ which is $\sigma$-finite with respect to $\HH^2$, we deduce thanks to \cite[Proposition 3.92]{AmFuPa00} that $u_3 \in SBV(\O_f)$. Finally, since $u_3$ is independent of $x_3$, we actually infer that $u_3 \in SBV(\omega)$. In addition, by uniqueness of the Lebesgue decomposition, it follows that 
$$\psi_\alpha= \partial_\alpha u_3,\quad 
[\vec u]_3 (\nu_{\vec u})_\alpha \mathcal H^2 \res J_{\vec u}=[u_3] (\nu_{u_3})_\alpha \mathcal H^2 \res [J_{u_3} \times (0,1)]$$
so that 
\begin{equation}\label{eq:Ju3}
J_{u_3} \times (0,1) \asubset J_{\vec u}.
\end{equation}
Integrating relation \eqref{eq:ualpha} with respect to $x_3$  yields
$$\bar u_\alpha(x'):= \int_0^{1} u_\alpha( x', x_3) \dd  x_3= u^+_\alpha( x', 0)  -\frac{1}{2}\partial_\alpha u_3(x')\quad \text{ for $\LL^2$-a.e. $x' \in \omega$},$$
from where \eqref{eq:ualphamed} follows.

\medskip
	
{\bf Step 3.} Let us prove that the approximate gradient of $u_3$, denoted by $\nabla u_3:=(\partial_1 u_3,\partial_2 u_3)$, and the averaged planar displacement $\bar{\vec u}:=(\bar u_1,\bar u_2)$ belong to $BD(\omega)$. For any $\varphi \in \C_c^\infty(\omega; \mathbb M^{2\times 2}_{\rm sym})$, according to the integration by parts formula in $BD$ (see \cite[Theorem 3.2]{B3}), we infer that
$$-\int_\omega \partial_\beta \varphi_{\alpha\beta} \bar u_\alpha \dd  x'= -\int_{\O_f} \partial_\beta\varphi_{\alpha \beta} u_\alpha \dd x
= \int_{\O_f} \varphi_{\alpha \beta} \dd E_{\alpha \beta}\vec u-\int_{\partial\O_f} \varphi_{\alpha \beta} u_\alpha \nu_\beta \dd\mathcal H^{2}.$$
Since $\varphi=0$ in a neighborhood of $\partial \omega \times (0,1)$ and $\nu=\pm e_3$ on $\omega \times \{0,1\}$, we get that the boundary term in the previous expression is zero. Therefore
\begin{equation}\label{eq:DistAVG}
-\int_\omega \partial_\beta \varphi_{\alpha\beta} \bar u_\alpha \dd  x'= \int_{\O_f} \varphi_{\alpha \beta} e_{\alpha\beta}(\vec u) \dd x + \int_{J_{\vec u}} \varphi_{\alpha\beta}  ([\vec u]\odot \nu_{\vec u})_{\alpha\beta}  \dd \mathcal H^2
\end{equation}
which shows that $\bar {\vec u} \in BD(\omega)$. According to slicing properties of $BD$ functions (see \cite[Proposition 3.4]{AmCoDM97}), for $\LL^1$-a.e. $x_3 \in (0,1)$, the function $(u_1(\cdot,x_3),u_2(\cdot,x_3)) \in BD(\omega)$ so that relation \eqref{eq:ualphamed} yields in turn that $\nabla u_3\in BD(\omega)$.

\medskip

{\bf Step 4.} We next establish that  $J_{\vec u} \cong (J_{\bar {\vec u}}Ê\cup J_{u_3}Ê\cup J_{\nabla u_3}) \times (0,1)$. To this aim, let us define the functions $\vec v:=(\bar u_1, \bar u_2, u_3)$ and $\vec g:= (\partial_1 u_3, \partial_2 u_3, 0)$. Since $u_3 \in SBV(\omega)$, $\bar{\vec u}Ê\in BD(\omega)$ and $\nabla u_3 \in BD(\omega)$, then clearly both $\vec v$, $\vec g \in BD(\O_f)$, and
\begin{equation}\label{eq:Jg}
J_{\vec g} = J_{\nabla u_3} \times (0,1).
\end{equation}
Moreover \cite[Proposition 3.92 (b)]{AmFuPa00} and \cite[Proposition 3.5]{AmCoDM97} imply that
$$J_{\bar{\vec u}} \cong \left \{ x' \in \omega: \limsup_{\varrho\to 0} \frac{|E\bar{\vec u}|(B'_\varrho( x'))}{\varrho} >0 \right \}, \; J_{u_3} \cong \left \{ x' \in \omega: \limsup_{\varrho\to 0} \frac{|Du_3|(B'_\varrho( x'))}{\varrho} >0 \right \},$$
and
$$J_{\vec v} \cong \Theta_{\vec  v} := \left \{ x \in \O_f: \limsup_{\varrho\to 0} \frac{|E\vec v|(B_\varrho( x))}{\varrho^2} >0 \right \},$$
where $B'_\varrho(x')$ stands for the two-dimensional open ball of center $x'$ and radius $\varrho$, while $B_\varrho(x)$ stands for the three-dimensional open ball of center $x$ and radius $\varrho$. Since $\vec v$ is independent of $x_3$, then
\begin{equation}\label{eq:Jv}
J_{\vec v} \cong (J_{\bar {\vec u}} \cup J_{u_3}) \times (0,1).
\end{equation}
According to \eqref{eq:Jg} and \eqref{eq:Jv}, it is thus enough to show that $J_{\vec u} \cong J_{\vec v}Ê\cup J_{\vec g}$. 

Let us also define the sets
\begin{align*}
\Theta_{\vec u} &:= \left \{ x \in \O_f: \limsup_{\varrho\to 0} \frac{|E\vec u|(B_\varrho( x))}{\varrho^2} >0 \right \}, \\
\Theta_{\vec g} &:= \left \{ x \in \O_f: \limsup_{\varrho\to 0} \frac{|E \vec g|(B_\varrho( x))}{\varrho^2} >0 \right \},
\end{align*}
and recall that, according again to  \cite[Proposition~3.5]{AmCoDM97}, $\Theta_{\vec u}\cong J_{\vec u}$ and $\Theta_{\vec g}\cong J_{\vec g}$. Using the expression of the displacement \eqref{eq:ualphamed}, we have $\vec u=\vec v + (\frac{1}{2}-x_3) \vec g$. Since $\vec u \in L^\infty(\O;\R^3)$, then $\vec vÊ\in L^\infty(\omega;\R^3)$ as well, and the previous relation yields $\vec g\in L^\infty(\omega;\R^3)$ with
$$\lim_{\varrho \to 0}\frac{1}{\varrho^2}\int_{B_\varrho(x)}|\vec g|\dd y =0 \quad \text{ for all }Êx\in \O_f.$$
Consequently since $E\vec u=E\vec v + (\frac{1}{2}-x_3) E\vec g -e_3 \odot \vec g$, we deduce that $\Omega_f \setminus (\Theta_{\vec v}\cup \Theta_{\vec g}) \subset \Omega_f \setminus \Theta_{\vec u}$ {\it i.e.}  $\Theta_{\vec u} \subset \Theta_{\vec v} \cup \Theta_{\vec g}$ and 
\begin{equation}\label{eq:incl1}
J_{\vec u} \asubset J_{\vec v}\cup J_{\vec g}.
\end{equation}
We now prove the converse inclusion. From the relations $\vec v=\vec u + (x_3-\frac{1}{2})\vec g$ and $(\frac{1}{2}-x_3)\vec g = \vec u - \vec v$, and the fact that $\vec g$ is independent of $x_3$, we similarly obtain that 
$\Theta_{\vec v} \subset \Theta_{\vec u} \cup \Theta_{\vec g}$ and $\Theta_{\vec g} \subset \Theta_{\vec u} \cup \Theta_{\vec v}$
which imply that
\begin{equation}\label{eq:incl2}
J_{\vec v}\setminus J_{\vec g} \asubset J_{\vec u},Ê\quad J_{\vec g}\setminus J_{\vec v} \asubset J_{\vec u}.
\end{equation}
It thus remains to prove that
\begin{equation}\label{eq:lastincl}
J_{\vec v}\cap J_{\vec g}\asubset J_{\vec u}.
\end{equation}
According to \eqref{eq:Ju3}, \eqref{eq:Jg} and \eqref{eq:Jv}, we get
$$(J_{\vec v}Ê\cap J_{\vec g}) \setminus J_{\vec u} \cong  ( [(J_{\bar {\vec u}} \cap J_{\nabla u_3}) \times (0,1)] \setminus J_{\vec u} \asubset  [(J_{\bar {\vec u}} \cap J_{\nabla u_3}) \setminus S_{u_3}] \times (0,1),$$
where we used that, since $u_3 \in SBV(\omega)$, then $J_{u_3}Ê\cong S_{u_3}$.
Assume by contradiction that 
\begin{equation}\label{eq:contrad}
\HH^2((J_{\vec v}Ê\cap J_{\vec g}) \setminus J_{\vec u})>0,
\end{equation}
then there is some $x=(x',x_3) \in (J_{\vec v}Ê\cap J_{\vec g}) \setminus J_{\vec u}$ with $x' \in (J_{\bar{\vec u}}Ê\cap J_{\nabla u_3}) \setminus S_{u_3}$ such that $\nu_{\bar{\vec u}}(x')=\pm \nu_{\nabla u_3}(x')$. Let us assume without loss of generality that $\nu_{\bar{\vec u}}(x')=\nu_{\nabla u_3}(x')=:\nu(x')$, the other case can be dealt with similarly. Since $x'$ is a Lebesgue point of $u_3$, then the one-sided Lebesgue limits of $u_3$ at $x'$ in the direction $\nu(x')$ are equal and coincide with its approximate limit. On the other hand, since $x' \in J_{\bar{\vec u}}Ê\cap J_{\nabla u_3}$, then the functions $\bar{\vec u}$ and $\nabla u_3$ admit one-sided Lebesgue limits at $x'$ in the direction $\nu(x')$. Next, from the expression \eqref{eq:ualphamed} of the displacement, we deduce that for all $\alpha \in \{1,2\}$, the functions $u_\alpha$ admit as well one-sided Lebesgue limits at $x$ in the direction $(\nu(x'),0)$. Gathering all previous informations, we get that the full displacement $\vec u$ admits one-sided Lebesgue limits at $x$ in the direction $(\nu(x'),0)$. Using the fact that $x \not\in J_{\vec u}$, we infer that necessarily $[\vec u](x)=0$, and thus, using again \eqref{eq:ualphamed} yields
\begin{equation}\label{eq:saut}
[\bar u_\alpha](x')+\left(\frac12 - x_3\right)[\partial_\alpha u_3](x')=0 \quad \text{ for all }\alpha \in \{1,2\}.
\end{equation}
We observe that, by \eqref{eq:Jg} and \eqref{eq:Jv}, the sets $J_{\vec v}$ and $J_{\vec g}$ are invariant in the transverse direction, and consequently $(x',y_3) \in J_{\vec v}Ê\cap J_{\vec g}$ for any $y_3 \in (0,1)$. Therefore if $(x',y_3) \not\in J_{\vec u}$ for some $y_3 \neq x_3$, then reproducing the same argument than above   implies that
$$[\bar u_\alpha](x')+\left(\frac12 - y_3\right)[\partial_\alpha u_3](x')=0  \quad \text{ for all }\alpha \in \{1,2\}.$$
Subtracting the previous relation to \eqref{eq:saut} yields $[\bar u_\alpha](x')=[\partial_\alpha u_3](x')=0$ for all $\alpha \in \{1,2\}$, which is against the fact that $x' \in J_{\bar{\vec u}}Ê\cap J_{\nabla u_3}$. As a consequence, $(x',y_3) \in J_{\vec u}$ for all $y_3 \in (0,1)$ with $y_3 \neq x_3$. In addition, since $x' \in  J_{\nabla u_3}$, there is some $\alpha\in \{1,2\}$ such that $[\partial_\alpha u_3](x') \neq 0$, and $x_3$ is therefore given by
$$x_3=\frac12 + \frac{[\bar u_\alpha](x')}{[\partial_\alpha u_3](x')}.$$
Consequently, we have proved that
$$(J_{\vec v}Ê\cap J_{\vec g}) \setminus J_{\vec u} \asubset \bigcup_{\alpha=1}^2 \left\{(x',x_3) : x' \in J_{\bar{\vec u}} \cap J_{\nabla u_3},\; [\partial_\alpha u_3](x')\neq 0,\; x_3=\frac12 + \frac{[\bar u_\alpha](x')}{[\partial_\alpha u_3](x')}\right\}=:A.$$
The set $A$ is Borel measurable, and, for each $x' \in J_{\bar{\vec u}} \cap J_{\nabla u_3}$, its transverse section passing through $x'$, denoted by  $A^{x'}:=\{x_3 \in (0,1) : (x',x_3) \in A\}$ is reduced to at most two points. Since the set $J_{\bar{\vec u}} \cap J_{\nabla u_3}$ is countably $\HH^1$-rectifiable, \cite[Theorem 3.2.23]{Federer} ensures that $\HH^2\res ((J_{\bar{\vec u}} \cap J_{\nabla u_3}) \times (0,1)) = (\HH^1 \res (J_{\bar{\vec u}} \cap J_{\nabla u_3})) \otimes (\LL^1 \res (0,1))$, and Fubini's Theorem yields
$$\HH^2(A)=\int_{J_{\bar{\vec u}} \cap J_{\nabla u_3}}\LL^1(A^{x'})\dd \HH^1(x')=0,$$
which is against \eqref{eq:contrad}, and therefore completes the proof of \eqref{eq:lastincl}. Gathering \eqref{eq:incl1} -- \eqref{eq:lastincl}  leads to $J_{\vec u} \cong J_{\vec v}\cup J_{\vec g}$, and thus $J_{\vec u} \cong (J_{\bar{\vec u}}Ê\cup J_{u_3}Ê\cup J_{\nabla u_3}) \times (0,1)$.

\medskip

{\bf Step 5.} We complete the proof of the proposition by establishing that $\bar{\vec u}$ and $\nabla u_3$ are actually $SBD(\omega)$ functions. Indeed, since we know that $J_{\vec u} \cong \Gamma \times (0,1)$ for some countably $\HH^1$-rectifiable set $\Gamma \subset \omega$, equation \eqref{eq:DistAVG} reads
$$-\int_\omega \partial_\beta \varphi_{\alpha\beta} \bar u_\alpha \dd  x'
= \int_\omega \varphi_{\alpha \beta} \left (\int_0^1 e_{\alpha\beta}(\vec u)\dd x_3\right) \dd x'
+ \int_\Gamma \varphi_{\alpha\beta} \left (\int_0^1 ([\vec u]\odot \nu_\Gamma)_{\alpha\beta}\dd x_3 \right ) \dd \mathcal H^1,$$
which implies that $e_{\alpha \beta}(\bar{\vec u})=\int_0^1 e_{\alpha\beta}(\vec u)(\cdot,x_3)\dd x_3$ by uniqueness of the Lebesgue decomposition, 
and that the singular part of $E\bar{\vec u}$ is concentrated on a countably $\mathcal H^1$-rectifiable set.
It follows from \cite[Proposition 4.7]{AmCoDM97} that $\bar{\vec u}\in SBD(\omega)$ and 
the same can be said, therefore, first for $\nabla u_3$ and then for $(u^+_1(\cdot, 0), u^+_2(\cdot, 0))$.
\end{proof}

Propositions \ref{prop:compactness} and \ref{prop:properties} suggest one to define the limiting space of all kinematically admissible displacements by
\begin{eqnarray}
\mathcal A_{KL} & := & \Bigg\{\vec u  \in SBD(\O_f) : \|\vec u\|_{L^\infty(\Omega_f)} \leq M, \; u_3 \in SBV(\omega) \cap L^\infty(\omega)\nonumber\\
&& \text{with }Ê\nabla u_3 \in SBD(\omega) \cap L^\infty(\omega;\R^2),\nonumber\\
&& u_\alpha( x', x_3)= \bar u_\alpha( x') + \Big(\frac{1}{2}-x_3\Big)\partial_\alpha u_3( x') \text{ for $\LL^3$-a.e.~$x=(x',x_3) \in \O_f$,}\label{eq:expru}\\
&&\text{where $\bar{\vec u}:=(\bar u_1,\bar u_2) \in SBD(\omega) \cap L^\infty(\omega;\R^2)$,}\nonumber\\
&&\text{and  $J_{\vec u} \cong (J_{\bar{\vec u}}Ê\cup J_{u_3}Ê\cup J_{\nabla u_3}) \times (0,1)$} \Bigg\}.\label{eq:exprJu}
\end{eqnarray}

\subsection{$\Gamma$-limit in the film}

\noindent For each $\eps>0$, let us define the functionals $\mathcal E^f_\eps$ and $\mathcal E^f_0 :L^2(\O_f;\R^3) \to [0,+\infty]$ by
$$
\mathcal E_\eps^f(\vec u):=\left\{
\begin{array}{ll}
E_\eps (\vec u, \O_f) & \text{ if } \vec u \in SBD(\O_f) \text{ and }Ê\|\vec u\|_{L^\infty(\O_f)} \leq M,\\
+\infty & \text{ otherwise,}
\end{array}
\right.
$$
and
$$
\mathcal E_0^f(\vec u):=\left\{
\begin{array}{l}
\displaystyle \int_\omega\left [ \frac{\lambda_f \mu_f}{\lambda_f + 2\mu_f} e_{\alpha \alpha}(\bar{\vec u})e_{\beta \beta}(\bar{\vec u}) + \mu_f e_{\alpha\beta}(\bar{\vec u})e_{\alpha\beta}(\bar{\vec u}) \right ] \dd x'\\
\displaystyle \hspace{0.5cm}+ \frac{1}{12} \int_\omega\left [ \frac{\lambda_f \mu_f}{\lambda_f +2 \mu_f} e_{\alpha \alpha}(\nabla u_3)e_{\beta \beta}(\nabla u_3)+ \mu_f e_{\alpha\beta}(\nabla u_3)e_{\alpha\beta}(\nabla u_3) \right ] \dd x'\\
\displaystyle \hspace{0.5cm} + \kappa_f\HH^1(J_{\bar{\vec u}}Ê\cup J_{u_3}Ê\cup J_{\nabla u_3}) \hfill \text{ if } \vec u \in \mathcal A_{KL} ,\\
+\infty \hfill \text{ otherwise.}
\end{array}
\right.
$$

\begin{theorem}\label{BH1}
The sequence of functionals $(\mathcal E_\eps^f)_{\eps>0}$ $\Gamma$-converges to $\mathcal E_0^f$ with respect to the strong $L^2(\O_f;\R^3)$-topology.
\end{theorem}

\begin{proof}
{\bf Step 1.} We start by deriving a lower bound inequality, {\it i.e.}, for any $\vec u \in L^2(\O_f;\R^3)$ and any sequence $(\vec u_\eps)_{\eps>0}Ê\subset L^2(\O_f;\R^3)$ such that $\vec u_\eps \to \vec u$ strongly in $L^2(\O_f;\R^3)$, then
$$\liminf_{\eps \to 0} \mathcal E_{\varepsilon}^f (\vec u_\eps) \geq \mathcal E_0^f(\vec u).$$

If $\liminf_\eps \mathcal E^f_{\varepsilon} (\vec u_\eps)=+\infty$, the result is obvious. Otherwise, up to a subsequence, we can assume that 
$$\lim_{\eps \to 0} \mathcal E^f_{\eps} (\vec u_\eps)=\liminf_{\eps \to 0} \mathcal E^f_{\varepsilon} (\vec u_\eps)<\infty.$$
By virtue of the above energy bound, we can assume without loss of generality that the conclusions of Propositions \ref{prop:compactness} and \ref{prop:properties} hold so that $\vec u \in \mathcal A_{KL}$. Using a very similar argument than that used in the proof of the lower bound in Theorem \ref{BH0}, combined with the lower semicontinuity of the surface energy established in \eqref{eq:liminfSE}, we obtain that
$$\liminf_{\eps \to 0} E_{\eps} (\vec u_\eps, \O_f) \geq \int_{\O_f}\left [ \frac{\lambda_f \mu_f}{\lambda_f + 2\mu_f} e_{\alpha \alpha}(\vec u)e_{\beta \beta}(\vec u)
+ \mu_f e_{\alpha\beta}(\vec u)e_{\alpha\beta}(\vec u) \right] \dd x +\kappa_f \HH^2(J_{\vec u}\cap \O_f).$$

According to \eqref{eq:expru}, we get that
\begin{multline*}
 \int_{\O_f} e_{\alpha \beta}(\vec u) e_{\alpha \beta}(\vec u) \dd x
=  \int_{\O_f} \Big[e_{\alpha \beta}(\bar{\vec u})e_{\alpha \beta}(\bar{\vec u})+ 2\left(\frac{1}{2}-x_3\right) e_{\alpha \beta}(\bar{\vec u}) e_{\alpha \beta}(\nabla u_3)\\ 
+ \left(\frac{1}{2}-x_3\right)^2e_{\alpha \beta}(\nabla u_3)e_{\alpha \beta}(\nabla u_3)\Big] \dd x\\
=\int_\omega e_{\alpha \beta}(\bar{\vec u})e_{\alpha \beta}(\bar{\vec u})\, \dd x'
+  \frac{1}{12}\int_{\omega} e_{\alpha \beta}(\nabla u_3)e_{\alpha \beta}(\nabla u_3) \dd x',
\end{multline*}
and similarly for the other term
$$\int_{\O_f} e_{\alpha \alpha}(\vec u) e_{\beta \beta}(\vec u) \dd x
=\int_\omega e_{\alpha \alpha}(\bar{\vec u})e_{\beta \beta}(\bar{\vec u})\, \dd x'
+  \frac{1}{12}\int_{\omega} e_{\alpha \alpha}(\nabla u_3)e_{\beta \beta}(\nabla u_3) \dd x'.$$
Therefore \eqref{eq:exprJu} yields the announced energy lower bound.

\medskip

{\bf Step 2.} We next derive an upper bound through the construction of a recovery sequence, {\it i.e.}, for every $\vec u \in L^2(\O_f;\R^3)$, there exists a recovery sequence $({\vec u}^*_\eps)_{\eps>0} \subset L^2(\O_f;\R^3)$ such that ${\vec u}^*_\eps \to \vec u$ strongly in $L^2(\O_f;\R^3)$, and
$$\limsup_{\eps \to 0} \mathcal E^f_\eps({\vec u}^*_\eps) \leq \mathcal E_0^f(\vec u).$$

If $\vec u \not\in \mathcal A_{KL}$, then $\mathcal E_0^f(\vec u)=+\infty$ and the result is obvious. It therefore suffices to assume that $\vec u \in \mathcal A_{KL}$. We now define a recovery sequence $({\vec u}^*_\eps)_{\eps>0}$. For $\LL^3$-a.e. $x=(x',x_3) \in \O_f$ and all $\eps>0$, let
$${\vec u}^*_\varepsilon( x', x_3 ) =c_\eps \big (\vec u(x) + (0,0, \varepsilon^2 x_3 h_\eps( x') )\big ),$$
where $(h_\eps)_{\eps>0}$ is a sequence in $\C_c^\infty(\omega)$ such that
\begin{equation}\label{eq:heps}
h_\eps \to - \frac{\lambda_f}{\lambda_f+2\mu_f} \vec e_{\alpha \alpha}(\vec u)  \text{ in } L^2(\omega), \quad \lim_{\eps \to 0} \varepsilon \|\nabla h_{\eps}\|_{L^2(\omega)}= \lim_{\eps \to 0} \varepsilon \| h_{\eps}\|_{L^\infty(\omega)}= 0,
\end{equation}
and $c_\eps:= M/(M+\eps^2\|h_\eps\|_{L^\infty(\omega)})$. Clearly, ${\vec u}^*_{\varepsilon}\in SBD(\O_f)$ and $\|{\vec u}^*_\varepsilon\|_{L^\infty(\O_f)} \leq M$. Using \eqref{eq:expru2} we get that 
	\begin{align*}
		J_\eps ({\vec u}^*_\eps, \Omega_f)
		&= \frac{c_\eps^2}{2}\int_{\Omega_f} \big [\lambda_f \vec e_{\alpha\alpha}(\vec u) \vec e_{\beta \beta}(\vec u)
			+ 2\mu_f \vec e_{\alpha \beta}(\vec u) \vec e_{\alpha \beta}(\vec u)\big ] \dd x
		\\ & \quad
			+ \frac{c_\eps^2}{2\varepsilon^2} \int_{\Omega_f} \left [2\lambda_f \vec e_{\alpha \alpha}(\vec u) \varepsilon^2  h_\eps + \mu_f \eps^4x_3^2|\nabla h_\eps|^2 \right ] \dd x 
		\\ & \quad 
			+ \frac{c_\eps^2}{2\varepsilon^4} \int_{\Omega_f} (\lambda_f + 2 \mu_f)\varepsilon^4 
			 |h_\eps|^2 \dd x.
	\end{align*}
Thus, since $c_\eps \to 1$ and according to the convergence properties \eqref{eq:heps}, we get that
$$\lim_{\eps \to 0} J_{\varepsilon}({\vec u}^*_\eps, \Omega_f) = \frac{1}{2} \int_{\Omega_f} \left [ \frac{2\lambda_f\mu_f}{\lambda_f+2\mu_f} \vec e_{\alpha\alpha}(\vec u) \vec e_{\beta \beta}(\vec u)+ 2\mu_f \vec e_{\alpha \beta}(\vec u) \vec e_{\alpha \beta}(\vec u)\right ] \dd x .$$

Concerning the surface energy, since $J_{{\vec u}^*_\eps}Ê= J_{\vec u} \cong  (J_{\bar{\vec u}}\cup J_{u_3} \cup J_{\nabla u_3}) \times (0,1)$ it follows that
$$\int_{\Omega_f \cap J_{{\vec u}^*_\eps}}	\left |\left ( (\vecg\nu_{{\vec u}^*_\eps})', \frac{1}{\varepsilon}(\vecg\nu_{{\vec u}^*_\eps})_3 \right ) \right | \dd\mathcal H^2 =  \mathcal H^1(J_{\bar{\vec u}} \cup J_{u_3} \cup J_{\nabla u_3}),$$
which completes the proof of the upper bound.
\end{proof}

\section{Multifissuration: debonding and delamination vs transverse cracks}\label{sec:6}

\noindent In this section, we consider the full model of a film $\O_f$ deposited on a substrate $\O_s$ through a bonding layer $\O_b$, and we assume that both $\O_f$ and $\O_b$ can crack.

\subsection{The anti-plane case}

Following \cite{LBBMM}, it is assumed that the geometry is invariant in the direction $\vec e_2$, {\it i.e.}, $\omega=I \times \R$, where $I$ is a bounded open interval, and that the admissible displacements take the form
$$\vec u(x)=u(x_1,x_3)\vec e_2.$$
In this case the elastic energy reduces to
\begin{multline*}
\tilde J_\eps(u)=\frac{\mu_f}{2} \int_{I \times (0,1)}Ê(|\partial_1 u|^2 + \eps^{-2}|\partial_3 u|^2)\dd x_1 \dd x_3\\
 + \frac{\mu_b}{2} \int_{I \times (-1,0)}(\eps^2|\partial_1 u|^2 + |\partial_3 u|^2)\dd x_1 \dd x_3,
 \end{multline*}
and the total energy is given by
\begin{multline*}
\tilde E_\eps(u):=\tilde J_\eps(u)+ \kappa_f \int_{J_u \cap [I \times (0,1)] } \left |\left ( (  \nu_u)_1, \varepsilon^{-1} (\nu_u)_3 \right ) \right | \dd\HH^1\\
+ \kappa_b \int_{J_u \cap  [I \times [-1,0]] } \left |\left ( \varepsilon ( \nu_u)_1,  (\nu_u)_3 \right ) \right | \dd\HH^1.
\end{multline*}
The natural functional setting is to consider (scalar) displacements in the class
$$\tilde {\mathcal A}:=\{u \in SBV(I \times (-2,1)) : u=0 \;\LL^2\text{-a.e. in }I \times (-2,-1) \text{ and }Ê\|u\|_{L^\infty(I\times (0,1))}Ê\leq M\},$$
where $M>0$ is an arbitrary fixed constant.

In \cite{LBBMM}, the following one-dimensional energy, defined for all $u \in SBV(I)$, was proposed as an approximation of the previous two-dimensional energy
$$\tilde E_0(u):=\frac{\mu_f}{2} \int_I |u'|^2\dd x_1 + \frac{\mu_b}{2} \int_{I \setminus \Delta_u} |u|^2\dd x_1 + \kappa_f \#(J_u) + \kappa_b \LL^1(\Delta_u),$$
where $\Delta_u:=\{ |u| > \sqrt{2 \kappa_b/\mu_b}\}$ is the delamination set. An easy adaptation of the proof of \cite[Theorem A.1]{LBBBHM} justifies rigorously this conjecture through the following $\Gamma$-convergence type result.

\begin{theorem}\label{BH}
Let $u \in SBV(I)$, then
\begin{itemize}
\item for any sequence $(u_\eps)_{\eps>0} \subset\tilde{\mathcal A}$ satisfying $u_\eps \to u$ strongly in $L^2(I \times (0,1))$, then
$$\tilde E_0(u) \leq \liminf_{\eps \to 0}\tilde E_\eps(u_\eps);$$
\item there exists a recovery sequence $(u_\eps^*)_{\eps>0} \subset \tilde{\mathcal A}$ such that $u^*_\eps \to u$ strongly in $L^2(I \times (0,1))$, and
$$\tilde E_0(u) \geq \liminf_{\eps \to 0} \tilde E_\eps(u_\eps^*).$$
\end{itemize}
\end{theorem}

Let us observe that if $u_\eps$ is a sequence of minimizers of $\tilde E_\eps$ (under suitable loadings), the (characteristic function of the) delamination set $\Delta_u$ is constructed as the $L^1$-limit of the orthogonal projection of the jump sets $J_{u_\eps}$ onto the mid-surface $\{x_3=0\}$. In particular, the vertical cracks in the bonding layer do not contribute to delamination.

\subsection{The general case}

We conjecture that Theorem \ref{BH} can be extended to the general  three-dimensional vectorial case. In this situation, the space of kinematically admissible displacements is  given by
$$\A:=\Big\{ \vec u \in SBD(\O) :  \vec u=0 \; \LL^3\text{-a.e. on }Ê\O_s, \text{ and }Ê\|\vec u\|_{L^\infty(\O_f)}Ê\leq M\Big\}.$$

Let us define the energy functionals $\mathcal E_\eps$ and $\mathcal E_0 : L^2(\O;\R^3) \to [0,+\infty]$ by
\begin{equation}\label{eq:Eeps}
\mathcal E_\eps(\vec u):=
\begin{cases}
E_\eps(\vec u) & \text{ if } \vec u \in \A,\\
+\infty & \text{ otherwise,}
\end{cases}
\end{equation}
and$$
\mathcal E_0(\vec u):= \left\{
\begin{array}{l}
\displaystyle \int_\omega \left [ \frac{\lambda_f \mu_f}{\lambda_f +2 \mu_f} e_{\alpha \alpha}(\bar{\vec u})e_{\beta \beta}(\bar{\vec u}) + \mu_f e_{\alpha \beta}(\bar{\vec u}) e_{\alpha \beta}(\bar{\vec u})\right ] \dd x'	\\  
\displaystyle \hspace{0.5cm} + \frac{1}{12} \int_\omega\left [ \frac{\lambda_f \mu_f}{\lambda_f +2 \mu_f} e_{\alpha \alpha}(\nabla u_3)e_{\beta \beta}(\nabla u_3)+ \mu_f e_{\alpha\beta}(\nabla u_3)e_{\alpha\beta}(\nabla u_3) \right ] \dd x' \\
\displaystyle \hspace{0.5cm}+ \frac{\mu_b}{2} \int_{\omega \setminus \Delta} |\bar{\vec u}|^2 \dd x'+ \kappa_f\HH^1(J_{\bar{\vec u}}Ê\cup J_{u_3}Ê\cup J_{\nabla u_3})+ \kappa_b \LL^2(\Delta) \hfill  \text{ if }\vec u \in \mathcal A_{KL},\\
+\infty \hfill \text{ otherwise},
\end{array}
\right.
$$
where the delamination set is defined by 
\begin{equation}\label{eq:delamination}
\Delta:=\left\{x' \in \omega : | \bar{\vec u}(x')| > \sqrt{\frac{2 \kappa_b}{\mu_b}}\right\} \cup \{x' \in \omega : u_3 \neq 0\}.
\end{equation}

We expect $\mathcal E_0$ to be the $\Gamma$-limit
of $\mathcal E_\eps$ as $\eps\to 0$,
but have been unable to prove the  
corresponding lower bound inequality:

\begin{conjecture}
If $\vec u \in L^2(\O;\R^3)$
and $(\vec u_\eps)_{\eps>0} \subset L^2(\O;\R^3)$ is
 any sequence converging strongly to $\vec u$ in $L^2(\O_f;\R^3)$, then
$$\mathcal E_0(\vec u)Ê\leq \liminf_{\eps \to 0}Ê\mathcal E_\eps (\vec u_\eps).$$
\end{conjecture}

Our aim here is only to prove the $\Gamma$-$\limsup$ inequality
and 
to present some partial results and techniques
which 
could be relevant in future investigations of this problem.

\begin{proposition}\label{pr:UB}
For every $\vec u \in L^2(\O;\R^3)$, there exists a sequence $({\vec u}^*_\eps)_{\eps>0}Ê\subset L^2(\O;\R^3)$ such that
${\vec u}^*_\eps \to \vec u$ strongly in $L^2(\O_f;\R^3)$, and
$$\mathcal E_0(\vec u)Ê\geq \limsup_{\eps \to 0}Ê\mathcal E_\eps ({\vec u}^*_\eps).$$
\end{proposition}

\begin{proof}
If $\vec u \not\in \A_{KL}$, then $\mathcal E_0(\vec u)=+\infty$ and there is nothing to prove. Therefore, we assume from now on that  $\vec u \in \A_{KL}$. 

{\bf Step 1.} In order to define the recovery sequence, we need several approximation steps. We start by approximating the delamination set defined in \eqref{eq:delamination}
by a sequence of sets of finite perimeter. Let 
	{$(\rho_m)_{m\in \N}$}
 be a standard sequence of mollifiers in $\R^2$, and set 
 	{$\chi_m:=\rho_m \ast \chi_{\Delta}$.}
  We know that 
  {$\chi_m \to \chi_\Delta$}
   strongly in $L^1(\omega)$. Set
$$
	{\delta_m:= \sqrt{\|\chi_m - \chi_{\Delta}\|_{L^1(\omega)}} \to 0.}
$$
By the coarea formula \cite[Theorem 3.40]{AmFuPa00}, for every $m \in \N$ large enough, there exists 
{$\frac{1}{2} \leq  t_m \leq 1-\delta_m$}
such that
{$$\Delta_m := \{ x'\in \omega: \chi_m(x') > t_m\}$$}
has finite perimeter. We claim that 
\begin{equation}\label{eq:chi}
	{	\chi_{\Delta_m}Ê}
\to \chi_{\Delta} \text{ in }L^1(\omega).
\end{equation}
Indeed,
{\begin{align*}
\LL^2(\Delta_m \setminus \Delta ) 
	&\leq \frac{1}{t_m} \int_{\Delta_m\setminus \Delta} \chi_m(x') \dd x'
	\leq \frac{1}{t_m} \int_{\Delta_m\setminus \Delta} 
		|\chi_m-\chi_{\Delta}| \dd x'
		\to 0,
		\end{align*}}
and
\begin{eqnarray*}
\LL^2(\Delta \setminus \Delta_m) & \leq & \LL^2(\{ x'\in \Delta: \chi_{\Delta}( x')=1\ \text{and}\ \chi_m( x') \leq 1-\delta_m\}) \\
& \leq & \frac{1}{\delta_m} \int_{\Delta} |\chi_{\Delta}( x') - \chi_m( x')|\dd x'
	{	\leq \delta_m \to 0,}
\end{eqnarray*}
hence $\| \chi_{\Delta} - \chi_{\Delta_m}\|_{L^1(\omega)} =\LL^2(\Delta_m \setminus \Delta )+\LL^2(\Delta \setminus \Delta_m) \to 0$. 
		{ 	In addition, it is possible to find a sequence $\eps_m 
		\overset{m\to\infty}{\longrightarrow} 0$
		such that $\eps_m \mathcal H^1(\partial ^*\Delta_m) 
				\overset{m\to\infty}{\longrightarrow} 0$.
		With a slight abuse of notation, we refer to the sequences 
		$(\eps_m)$ and $(\Delta_m)$ 
		simply as $(\eps)$ and $(\Delta_\eps)$
		and henceforth assume that}
\begin{equation}\label{eq:heps2}
\lim_{\eps \to 0} \varepsilon \HH^1(\partial^*\Delta_\eps)=0.
\end{equation}

We next approximate the  displacement $\vec u$. Indeed, according to \cite[Theorem 3]{C2} (see also \cite[Theorem 3]{I}), there exists a sequence $(\bar{\vec u}_\eps)_{\eps>0} \in SBV(\omega;\R^2)$ such that $\bar{\vec u}_\eps \to \bar{\vec u}$ strongly in $L^2(\omega;\R^3)$, $e(\bar{\vec u}_\eps) \to e(\bar{\vec u})$ strongly in $L^2(\omega;\mathbb M^{2 \times 2}_{\rm sym})$, $\HH^1(J_{\bar{\vec u}_\eps} \setminus J_{\bar{\vec u}})+\HH^1(J_{\bar{\vec u}} \setminus J_{\bar{\vec u}_\eps})\to 0$, and $\|\bar{\vec u}_\eps\|_{L^\infty(\omega)} \leq \|\bar{\vec u}\|_{L^\infty(\omega)}$. Let us define for a.e. $x' \in \omega$ and all $x_3 \in (0,1)$,
$$(u_\eps)_\alpha(x',x_3):=(\bar u_\eps)_\alpha(x') + \left(\frac12 -x_3 \right)\partial_\alpha u_3(x'), \quad (u_\eps)_3(x',x_3):=u_3(x')$$
so that $\vec u_\eps \in SBD(\O_f)$, and $\|\vec u_\eps\|_{L^\infty(\O_f)} \leq \|\vec u\|_{L^\infty(\O_f)} \leq M$. 

As in the proof of Theorem \ref{BH1}, we consider a sequence $(h_\eps)_{\eps>0} \subset \C_c^\infty(\omega)$ satisfying \eqref{eq:heps}. 

We now define the recovery sequence by setting, for all $\eps>0$ and for $\LL^3$-a.e. $x=(x',x_3) \in \O$,
$$		
{\vec u}^*_\eps( x', x_3 ) =
\begin{cases}
\displaystyle c_\eps \big( \vec u_\eps(x) +\big (0,0,\varepsilon^2 x_3 h_\eps( x') \big )&  \text{if}\ (x',x_3) \in \O_f,\\
c_\eps (x_3+1) (\bar{\vec u}_\eps(x'),0) & \text{if}\  (x',x_3) \in (\omega \setminus \Delta_\eps) \times [-1,0], \\
0 & \text{if}\ (x',x_3) \in ( \Delta_\eps \times [-1,0]) \cup \O_s,
\end{cases}
$$
where  $c_\eps=\frac{M}{M+\eps^2\|h_\eps\|_{L^\infty(\omega)}}$. Since the set $\Delta_\eps$ has finite perimeter in $\omega$ and $\bar{\vec u}_\eps \in SBV(\omega;\R^2)$, then $\bar{\vec u}_\epsÊ\chi_{\omega \setminus \Delta_\eps} \in SBV(\omega;\R^2)$, and thus ${\vec u}^*_\eps \in SBD(\O)$ and ${\vec u}^*_\eps =0$ $\LL^3$-a.e. in $\O_s$. In addition,
 the fact that $\|\vec u_\eps\|_{L^\infty(\O_f)}Ê\leq M$ yields $\|{\vec u}^*_\eps \|_{L^\infty(\O_f)}Ê\leq M$ as well so that ${\vec u}^*_\eps \in \A$. The sequence $({\vec u}^*_\eps)_{\eps>0}$ is thus admissible, and clearly ${\vec u}^*_\eps \to \vec u$ strongly in $L^2(\O_f;\R^3)$. 

\medskip

{\bf Step 2.} Using the convergence properties of $\bar{ \vec u}_\eps$, a similar argument than in the proof of Theorem \ref{BH1} leads to
$$\limsup_{\eps \to 0}E_\eps({\vec u}^*_\eps,\O_f) \leq \mathcal E_0^f(\vec u).$$
It thus remains to compute the energy associated to this sequence in the bonding layer. First, the bulk energy in the bonding layer gives
\begin{multline*} 
J_{\eps}({\vec u}^*_\eps, \Omega_b)= \frac{c_\eps^2\eps^2}{2} \int_{(\omega \setminus \Delta_\eps) \times (-1,0)} (x_3+1)^2\Big [\lambda_b  e_{\alpha \alpha}(\bar{\vec u}_\eps )  e_{\beta \beta} (\bar{\vec u}_\eps )
+2\mu_b  e_{\alpha \beta} (\bar{\vec u}_\eps )  e_{\alpha \beta} (\bar{\vec u}_\eps )	\Big ] \dd x \\
+\frac{c_\eps^2 \mu_b}{2} \int_{\omega \setminus \Delta_\eps} |\bar{\vec u}_\eps|^2\dd x' \to \frac{\mu_b}{2}\int_{\omega \setminus \Delta}Ê|\bar{\vec u}|^2\dd x'.
 \end{multline*}

Concerning the surface energy in the bonding layer, we first observe that  for each $\eps>0$,
$$J_{{\vec u}^*_\eps} \cap \Omega_b\subset	\Big [ J_{\bar{\vec u}_\eps} \times [-1,0] \Big ]\cup \Big [ \Delta_\eps \times \{0\} \Big ]	\cup \Big [ (\{(u_3,\nabla u_3) \neq 0\} \setminus \Delta_\eps) \times \{0\} \Big ]	 \cup \Big [\partial^*\Delta_\eps \times [-1,0] \Big],$$
where $\partial^* \Delta_\eps$ stands for the reduced boundary of $\Delta_\eps$ \cite[Definition 3.54]{AmFuPa00}. Let us observe that $\omega \setminus \Delta \subset \{u_3=0\} \asubset \{(u_3,\nabla u_3)=0\}$ since, by locality of the approximate gradient, $\nabla u_3=0$ $\LL^2$-a.e. in $\{u_3=0\}$ (see \cite[Proposition 3.73 (c)]{AmFuPa00}).
Then
\begin{multline*}
\limsup_{\eps \to 0} \int_{J_{{\vec u}^*_\eps} \cap \Omega_b} \left |\left ( \eps ( \vecg \nu_{{\vec u}^*_\eps})',  (\vecg\nu_{{\vec u}^*_\eps})_3 \right ) \right | \dd\HH^2\\
\leq \limsup_{\eps \to 0} \Big [\eps \HH^1(J_{\bar{\vec u}_\eps}) + \LL^2 (\Delta_\eps) +\LL^2(\{(u_3,\nabla u_3) \neq 0\} \setminus \Delta_\eps)+ \eps \HH^1 (\partial^* \Delta_\eps) \Big ] = \LL^2(\Delta),
\end{multline*}
thanks to \eqref{eq:delamination}, \eqref{eq:chi} and \eqref{eq:heps2}.
\end{proof}

\subsubsection{Partial results for the lower bound}

Let $\vec u \in L^2(\O;\R^3)$, and $(\vec u_\eps)_{\eps>0} \subset L^2(\O;\R^3)$ be a sequence such that $\vec u_\eps \to \vec u$ strongly in $L^2(\O_f;\R^3)$. If $\liminf_\eps \mathcal E_\eps (\vec u_\eps)=+\infty$ there is nothing to prove. Otherwise by \eqref{eq:Eeps}, up to a subsequence, we can assume without loss of generality that $(\vec u_\eps)_{\eps>0} \subset \mathcal A$, and that
\begin{equation}\label{eq:nrjbound}
\sup_{\eps>0}E_\eps (\vec u_\eps)<+\infty.
\end{equation}
As a consequence, all the compactness results in the film $\O_f$ established in section \ref{sec:compactness} hold. In particular, Propositions \ref{prop:compactness} and \ref{prop:properties} show that $\vec u \in \A_{KL}$, 
	and 
	the lower bound established in Theorem \ref{BH1} 
	yields the terms in $\mathcal E_0(\vec u)$ corresponding
	to the energy in $\Omega_f$.
The main problem is to deal with the bonding layer. Following the scalar case treated in \cite{LBBBHM}, it is enough to show that the energy in $\Omega_b$ is bounded from below by some  functional where the  delamination set is replaced by a function $\theta \in L^\infty(\omega;[0,1])$, which can be interpreted as a delamination volume fraction density. On $\{\theta=1\}$, the film is entirely debonded from the substrate, while on $\{\theta=0\}$ it continuously accommodates the prescribed zero displacement on the substrate exactly as in the Sobolev case (Theorem \ref{BH0}). All intermediate states are contained in the set $\{0<\theta<1\}$.

\begin{proposition}\label{prop:theta}
Assume there exists $\theta \in L^\infty(\omega;[0,1])$ such that $(1-\theta)u_3=0$ $\LL^2$-a.e. in $\omega$, and
\begin{align}\label{eq:good}
	\frac{\mu_b}{2}\int_{\omega}(1-\theta)Ê|\bar{\vec u}|^2\dd x' + \kappa_b \int_\omega \theta\dd x' \leq \liminf_{\eps \to0}ÊE_\eps(\vec u_\eps,\O_b).
\end{align}
Then
$$\frac{\mu_b}{2}\int_{\omega \setminus \Delta}Ê|\bar{\vec u}|^2\dd x' + \kappa_b \LL^2(\Delta) \leq \liminf_{\eps \to 0}ÊE_\eps(\vec u_\eps,\O_b),$$
where $\Delta$ be the delamination set defined in \eqref{eq:delamination} 
\end{proposition}

\begin{proof}
By assumption, we have that
$$\int_\omega \min_{\{\eta \in [0,1] : (1-\eta)u_3(x')=0\}} \left(\frac{\mu_b}{2}(1-\eta)|\bar{\vec u}(x')|^2 + \kappa_b \eta\right)\dd x' \leq\liminf_{\eps \to 0}ÊE_\eps(\vec u_\eps,\O_b).$$
The result follows by solving  the  pointwise minimization problem  explicitly. 
\end{proof}

The main point is to construct such a function $\theta$. As in the scalar case \cite{LBBBHM}, $\theta$ is supposed to be obtained as the $L^\infty(\omega)$-weak* limit of a sequence $(\chi_{\Delta_\eps})_{\eps>0}$ of suitable measurable sets $\Delta_\eps \subset \omega$. 
However, it is unclear what is the right notion of an $\eps$-delamination set $\Delta_\eps$ in the vectorial case.
In particular, the following example shows that vertical cracks in the bonding layer
cannot be neglected, so it is not enough to define $\Delta_\eps$ as
the orthogonal projection of $J_{\vec u_\eps}$ onto the mid-plane $\omega \times \{0\}$,
as in the anti-plane and in the Sobolev case (Thm.\ \ref{BH}, \cite[Prop.\ B.2]{LBBBHM}, and Thm\  \ref{BH0}).

\begin{example}[Microstructure example] \label{ex:micro}
{\rm Suppose that $\omega=(0,1)^2$ and $\varepsilon= \frac{1}{2N}$ for some $N\in \N$.
In the film, set 
$$\vec u_\eps(x)=\vec u(x)=(0,\ell,0) \quad \text{ for all } x \in \O_f.$$
In $\Omega_b$ set, for each $i=0, \ldots, N-1$
and all $2i\varepsilon \leq x_2 \leq (2i+2)\eps$, $-1\leq x_3\leq 0< x_1< 1$,
\begin{align*}
	u_\eps(x_1,x_2,x_3)=\left (0, \ell(1+x_3), \ell \eps v \left(\frac{x_2-2i\eps}{\eps},1+x_3\right) 
	\right ),
\end{align*}
where $v \in H^1\Big ((0,2)\times(0,1)\Big )$ is any function such that $v(s,0)=v(s,1)=0\ \forall s\in [0,1]$ and
$$q:=\fint_{s=0}^2\int_{t=0}^1 \big(( 1+\partial_s v )^2 + 2\partial_t v^2\big) \dd s\dd t < 1.$$
If $\Delta_\eps$ is defined as $\pi(J_{\vec u_\eps}Ê\cap \O)$, then 
\begin{equation}\label{1706}
\int_{\O_b} \big(2\mu_b e_{\alpha 3}(\vec u_\eps)e_{\alpha 3}(\vec u_\eps) + \eps^{-2}\mu_b e_{33}(\vec u_\eps)e_{33}(\vec u_\eps)\big)\dd x+ \kappa_b \LL^2(\Delta_\eps)
=\frac{q\mu_b \ell^2}{2}.
\end{equation}
On the other hand, if $\Delta=\{Ê|\bar{\vec u}| >\sqrt{2\kappa_b/\mu_b}\}$ is the expected limit delamination set, then
\begin{equation}\label{1705}
 \int_{\omega \setminus \Delta} \frac{\mu_b}{2}u_\alpha u_\alpha \dd x' + \kappa_b \LL^2(\Delta) = 
\left\{
\begin{array}{lll}
\frac{\mu_b \ell^2}{2} & \text{ if } & \ell \leq \sqrt{\frac{2\kappa_b}{\mu_b}},\\
\kappa_b & \text{ if } & \ell > \sqrt{\frac{2\kappa_b}{\mu_b}}.
\end{array}
\right.
\end{equation}
Choosing $\ell\in \displaystyle \left (\sqrt{\frac{2\kappa_b}{\mu_b}}, \sqrt{\frac{2\kappa_b}{q\mu_b}}\right )$  shows that \eqref{1705} would not be a lower bound for \eqref{1706}. 
}
\qed
\end{example}

Regardless of the notion of an $\eps$-delamination set $\Delta_\eps$  one
tries to define, it is convenient to impose that it should contain the set 
$$P_\eps:= \pi(J_{\vec u_\eps}Ê\cap\O_f),$$
where $\pi : \R^3 \to \R^2$, $\pi(x):=x'$,  is the orthogonal projection onto $\R^2 \times \{0\}$.
On the one hand, there is no loss of generality in doing this, since it converges to 
a Lebesgue negligible set.
Indeed, according to the coarea formula (see \cite[Theorem 2.93]{AmFuPa00}) and the surface energy bound \eqref{eq:nrjbound} in the film, we have 
$$\LL^2(P_\eps) \leq \int_{\R^2} \HH^0(J_{\vec u_\eps} \cap \O_f \cap \pi^{-1}(x')) \dd x' =\int_{J_{\vec u_\eps}\cap \O_f} |( \nu_{\vec u_\eps})_3 | \dd\HH^2  \leq C \eps \to 0.$$
On the other hand, excluding $P_\eps$ enables one to slightly improve the convergences in the film, as in the following lemma which proves the convergence of the planar gradient of the anti-plane displacement.

It will be assumed henceforth that 
 $u_\eps \in SBV^2(\O;\R^3)$ and that  $J_{\vec u_\eps}$ is closed in $\O$ and contained in a finite union of closed connected pieces of $\C^1$ hypersurfaces. In doing this no generality is lost, thanks to the density result in $SBD$ of \cite[Thm.\ 1]{C2}. In particular, we have that $\vec u_\eps \in H^1((\omega \setminus P_\eps) \times (0,1);\R^3)$.

\begin{lemma}\label{lem:conv-impr}
Let $(\Delta_\eps)_{\eps>0}$ be a sequence of closed sets be such that $P_\eps \subset \Delta_\eps$ for each $\eps>0$. Assume that there exists a function $\theta \in L^\infty(\omega;[0,1])$ such that $\chi_{\Delta_\eps}Ê\weakcs \theta$ weakly* in $L^\infty(\omega)$, and $(1-\theta)u_3=0$ $\LL^2$-a.e. in $\omega$. Then 
$$\chi_{\omega \setminus \Delta_\eps} \partial_\alpha (u_\eps)_3 \weakcs 0 \quad \text{  weakly* in }L^2(\omega;H^{-1}(0,1)).$$
\end{lemma}

\begin{proof}
First note that
for $\LL^2$-a.e. $x' \not\in P_\eps$ and  $\LL^1$-a.e. $x_3 \in (0,1)$
\begin{align*}
	\zeta_\alpha^\eps(x)
	&:=\int_0^{x_3} \partial_\alpha (u_\eps)_3(x',s)\dd s +(u_\eps)_\alpha(x) - (u_\eps)_\alpha^+(x',0)
	\\ &= \int_0^{x_3} [\partial_\alpha (u_\eps)_3(x',s) +\partial_3 (u_\eps)_\alpha(x',s) ]\dd s = 2 \int_0^{x_3} e_{\alpha 3}(\vec u_\eps)(x',s)\dd s.
\end{align*}
Thanks to the bulk energy bound \eqref{eq:nrjbound} in the film (see also \eqref{eq:bound_alpha3}), we have that
 \begin{align} \label{eq:lowerbd1}
 \|\zeta_\alpha^\eps\|_{L^2((\omega \setminus \Delta_\eps) \times (0,1))} 
\leq 2\|e_{\alpha 3}(\vec u_\eps)\|_{L^2(\Omega_f)} \leq C\eps \to 0.
\end{align}
Integrating \eqref{eq:lowerbd1} we obtain that also $\|\bar \zeta_\alpha^\eps\|_{L^2(\omega \setminus \Delta_\eps)} \to 0$, where
$$
\bar \zeta_\alpha^\eps(x'):=\int_0^1 \zeta_\alpha^\eps(x',x_3)\dd x_3
=\int_0^1 \int_0^{x_3} \partial_\alpha (u_\eps)_3(x',s)\dd s\dd x_3 +(\bar u_\eps)_\alpha(x') - (u_\eps)_\alpha^+(x',0)$$
and $(\bar u_\eps)_\alpha(x'):=	\int_0^1( u_\eps)_\alpha(x',x_3)\dd x_3$.
As a consequence,
\begin{eqnarray}\label{eq:1459}
(u_\eps)_\alpha(x) & = & (u_\eps)_\alpha^+(x',0)- \int_0^{x_3} \partial_\alpha (u_\eps)_3(x',s)\dd s +\zeta_\alpha^\eps(x)\nonumber\\
& = & (\bar u_\eps)_\alpha(x') + \int_0^1 \int_{0}^{x_3} \partial_\alpha (u_\eps)_3(x',s)\dd s\dd x_3- \int_{0}^{x_3} \partial_\alpha (u_\eps)_3(x',s)\dd s +\eta_\alpha^\eps(x),\end{eqnarray}
where $\|\eta_\alpha^\eps\|_{L^2((\omega \setminus \Delta_\eps) \times (0,1))} \to 0$. 

On the other hand, for $\LL^3$-a.e. $x \in \O_f$, let us define the sequences
\begin{eqnarray*}
g_\alpha^\eps(x',x_3) & :=& \chi_{\omega \setminus \Delta_\eps}(x') \int_{0}^{x_3} \partial_\alpha(u_\eps)_3(x',s)\dd s,\\
 \bar g_\alpha^\eps(x') & := & \chi_{\omega \setminus \Delta_\eps}(x')\int_0^1 \int_{0}^{x_3} \partial_\alpha(u_\eps)_3(x',s)\dd s\dd x_3.
 \end{eqnarray*}
From \eqref{eq:lowerbd1} and the {\it a priori} bound $\|\vec u_\eps\|_{L^\infty(\O_f)} \leq M$, 
we get $\|g_\alpha^\eps\|_{L^2(\O_f)} \leq C$ for some constant $C>0$ independent of $\eps$. Therefore, up to a subsequence, $g_\alpha^\eps \weakc g_\alpha$ weakly in $L^2(\O_f)$ for some $g_\alpha \in L^2(\O_f)$. In addition,  $\bar g_\alpha^\eps \weakc \bar g_\alpha$ weakly in $L^2(\omega)$, where $\bar g_\alpha(x'):=\int_0^1 g_\alpha(x',x_3)\dd x_3$. 

Multiplying \eqref{eq:1459} by $\chi_{\omega \setminus \Delta_\eps}$ leads to 
$$(u_\eps)_\alpha(x) \chi_{\omega \setminus \Delta_\eps}(x') = (\bar u_\eps)_\alpha(x')\chi_{\omega \setminus \Delta_\eps} (x') + \bar g_\alpha^\eps(x') - g_\alpha^\eps(x) + \tilde\eta_\alpha^\eps(x),$$
where $\|\tilde \eta_\alpha^\eps\|_{L^2(\O_f)} \to 0$.
Passing to the limit as $\eps \to 0$ finally yields
$$(1-\theta(x')) (u_\alpha(x) -\bar u_\alpha(x')) = \bar g_\alpha(x') - g_\alpha(x),$$
and according to the structure \eqref{eq:expru} of planar displacements, we deduce that
$$\left(\frac{1}{2}-x_3\right)(1-\theta(x'))\partial_\alpha u_3(x') = \bar g_\alpha(x') - g_\alpha(x).$$
Since by assumption $u_3=0$ $\LL^2$-a.e. in $\{\theta<1\}$, we get by locality of approximate gradients of $SBV$ functions (see \cite[Proposition 3.73 (c)]{AmFuPa00}), that $\nabla u_3=0$ $\LL^2$-a.e. in $\{\theta<1\}$, hence $g_\alpha(x)=\bar g_\alpha(x')$. As a consequence, $\chi_{\omega \setminus \Delta_\eps} \partial_\alpha (u_\eps)_3=D_3 g_\alpha^\eps  \weakcs D_3 g_\alpha=0$ weakly* in $L^2(\omega;H^{-1}(0,1))$.
\end{proof}

An alternative to the definition of $\Delta_\eps$ as the orthogonal projection of $J_{\vec u_\eps}$ onto $\omega\times \{0\}$ is to consider its projection along certain 
almost-vertical oblique directions.
Define the unit vectors
$$\vecg\xi^\pm=  \frac{1}{\sqrt{2}} (\pm 1,0,1), \quad \vecg \eta^\pm=  \frac{1}{\sqrt{2}} (0,\pm 1,1),$$
and their rescaled versions
$$\vecg\xi^\pm_\eps :=  \frac{1}{\sqrt{2}}\left(\pm 1,0,\eps^{-1} \right), \quad \vecg \eta^\pm_\eps :=  \frac{1}{\sqrt{2}} \left(0, \pm 1, \eps^{-1} \right).$$
Denote by $\pi_{\vecg\xi^\pm_\eps}$ (resp. $\pi_{\vecg\eta^\pm_\eps}$) $: \R^3 \to \R^2$ the projection onto $\{x_3=0\}$ parallel to the vector $\vecg\xi^\pm_\eps$ (resp. $\vecg\eta^\pm_\eps$), {\it i.e.}, for $x:=(x', 0)+t\vecg\xi^\pm_\eps$ (resp.  $x:=(x', 0)+t\vecg \eta^\pm_\eps$), then $\pi_{\vecg\xi^\pm_\eps} (x ):=x'$ (resp.  $\pi_{\vecg\eta^\pm_\eps} (x ):=x'$). 
Finally, consider the set
\begin{eqnarray*}
\Delta_\eps& := & \pi_{\vecg\xi^+_\eps} \big(J_{\vec u_\eps}\cap (\omega_\eps \times (-2,1))\big)\cup\pi_{\vecg\eta^+_\eps} \big(J_{\vec u_\eps}\cap (\omega_\eps \times (-2,1))\big)\\
&&  \cup \; \pi_{\vecg\xi^-_\eps} \big(J_{\vec u_\eps}\cap (\omega_\eps \times (-2,1))\big)\cup\pi_{\vecg\eta^-_\eps} \big(J_{\vec u_\eps}\cap (\omega_\eps \times (-2,1))\big) \cup P_\eps.
\end{eqnarray*}
where $\omega_\eps:=\{x' \in \omega : \dist(x', \partial \omega)> 2\eps\}$. Up to a subsequence, it can be assumed that
$$\chi_{\Delta_\eps}Ê\weakcs \theta \quad \text{weakly* in }ÊL^\infty(\omega)
\quad \text{for some } \theta \in L^\infty(\omega;[0,1]).$$ 

Using the decomposition  
\begin{eqnarray*}
\left |\big ( \eps (\nu_{\vec u_\eps})', ( \nu_{\vec u_\eps} )_3 \big )\right |^2 & = & \frac12 |\eps (\nu_{\vec u_\eps})_1 + (\nu_{\vec u_\eps})_3|^2 + \frac12 |\eps (\nu_{\vec u_\eps})_1 - (\nu_{\vec u_\eps})_3|^2 +\eps^2 |( \nu_{\vec u_\eps} )_2|^2\\
& = & \frac12 |\eps (\nu_{\vec u_\eps})_2 + (\nu_{\vec u_\eps})_3|^2 + \frac12 |\eps (\nu_{\vec u_\eps})_2 - (\nu_{\vec u_\eps})_3|^2 +\eps^2 |( \nu_{\vec u_\eps} )_1 |^2,
\end{eqnarray*}
it is possible to prove that
$$\liminf_{\eps \to 0}\int_{J_{\vec u_\eps}Ê\cap \O_b}\left |\big ( \eps (\nu_{\vec u_\eps})', ( \nu_{\vec u_\eps} )_3 \big )\right |\dd \HH^2 \geq \frac18 \int_\omega \theta\dd x',$$
which shows that $\|\theta\|_{L^1(\omega)}$ 
is controlled (up to a multiplicative constant) by the fracture energy in the bonding layer.
The constant $1/8$, however, is not optimal, 
since in order to obtain \eqref{eq:good} that prefactor
should not be present.
In most situations ({\it e.g.}\ if the  sets $\pi(J_{\vec u_\eps}\cap \Omega_b)$ have uniformly bounded
perimeters)
it should be possible to obtain the optimal lower bound,
but there are pathological cases (such as the microstructure Example \ref{ex:micro})
where $\int_\omega \theta \dd x'$ is larger than the fracture energy on the left-hand side 
(because each vertical crack is counted twice in $\Delta_\eps$, which is defined as the
union of all the oblique projections).

Be it as it may, by including in $\Delta_\eps$ 
the oblique projections of the cracks inside the bonding layer,
one is able to obtain an optimal estimate 
for the elastic energy required by the body to accomodate
the strain mismatch between the deformations in the film
and in the rigid substrate. Before proving this final estimate,  we need two preliminary technical results concerning sections of $BD$-functions along the oblique directions defined above. For $\LL^2$-a.e.\ $x'\in \omega_\eps$ and $\LL^1$-a.e. $t \in (-2 \sqrt 2 \eps,\sqrt 2 \eps)$,  define the functions
$$(\vec u_\eps)^{ x'}_{\vecg\xi^\pm_\eps}(t):= \vec u_\eps \Big ( ( x', 0)+t\vecg\xi^\pm_\eps\Big )\cdot \vecg\xi^\pm_\eps,\quad (\vec u_\eps)^{ x'}_{\vecg\eta^\pm_\eps}(t):= \vec u_\eps \Big ( ( x', 0)+t\vecg\eta^\pm_\eps\Big )\cdot \vecg\eta^\pm_\eps.$$

\begin{lemma}\label{lem:Sobolev}
For $\LL^2$-a.e. $x' \in \omega_\eps \setminus \Delta_\eps$, we have
$$(\vec u_\eps)^{ x'}_{\vecg\xi^\pm_\eps} \in H^1(-\sqrt 2 \eps,\sqrt 2 \eps) \text{ and } (\vec u_\eps)^{ x'}_{\vecg\eta^\pm_\eps}\in H^1(-\sqrt 2 \eps,\sqrt 2 \eps),$$
with  $(\vec u_\eps)^{ x'}_{\vecg\xi^\pm_\eps} (-\sqrt 2 \eps)=(\vec u_\eps)^{ x'}_{\vecg\eta^\pm_\eps}(-\sqrt 2 \eps)=0$, and
$$x_3 \mapsto (u_\eps)_3(x',x_3) \in H^1(0,1).$$
\end{lemma}

\begin{proof}
Let us denote by 
$$\Pi_{\vec\xi^{\pm}_\eps}:=\{\zeta \in \R^3 : \; \zeta \cdot \vec\xi^{\pm}_\eps=0\}$$
the  plane orthogonal to $\vec\xi^{\pm}_\eps$ passing through the origin, and, for $y \in \Pi_{\vec\xi^{\pm}_\eps}$, we define 
$$\O^y_{\vecg\xi^{\pm}_\eps}:=\{ t \in \R : y+t\vec\xi^{\pm}_\eps \in \O_f\}.$$
According to slicing properties of functions of bounded deformations (see \cite[Theorem 4.5]{AmCoDM97}), we know that for $\HH^2$-a.e. $y \in \Pi_{\vec \xi_\eps^{\pm}}$, the function
\begin{equation*}
t \mapsto \vec u_\eps \big ( y+t\vecg\xi^{\pm}_\eps\big )\cdot \vecg\xi^{\pm}_\eps
\ \text{belongs to}\ SBV^2\big(\O^y_{\vec \xi_\eps^{\pm }}\big),
\end{equation*}
and its jump set is contained in 
$$\{t \in \O^y_{\vecg\xi^{\pm}_\eps} : y+t \vec \xi_\eps^\pm \in J_{\vec u_\eps}\} .$$
Let us denote by $N_{\vec \xi^\pm_\eps} \subset  \Pi_{\vec \xi_\eps^{\pm}}$ the exceptional set of zero $\HH^2$ measure on which the previous properties fail. Since $\pi_{\vec \xi^\pm_\eps}$ are Lipschitz functions, it follows that the sets $Z_{\vec \xi^\pm_\eps}:=\pi_{\vec \xi^\pm_\eps}(N_{\vec \xi^\pm_\eps}) \subset \omega$ are $\LL^2$-negligible as well. Consequently, for all $x' \in \omega_\eps \setminus Z_{\vec \xi^\pm_\eps}$ (and thus for $\LL^2$-a.e. $x'\in \omega_\eps$), we have that
$$(\vec u_\eps)^{ x'}_{\xi_\eps^{\pm}}\in SBV^2(-2 \sqrt 2 \eps,\sqrt 2 \eps),$$
and its jump set is contained in 
$$\{t \in (- 2\sqrt 2 \eps,\sqrt 2 \eps) : (x',0)+t \vec \xi_\eps^\pm \in J_{\vec u_\eps}\}.$$

By definition of the set $\Delta_\eps$, if $x' \in \omega_\eps \setminus  \Delta_\eps$ then $(x',0) + t \vec \xi_\eps^{\pm} \not \in J_{\vec u_\eps}Ê\cap [\omega_\eps \times (-2,1)]$ for all $t \in (-2\sqrt 2 \eps,\sqrt 2 \eps)$, and therefore $(\vec u_\eps)^{x'}_{\xi_\eps^{\pm}} \in H^1(-2\sqrt 2 \eps,\sqrt 2 \eps)$ for $\LL^2$-a.e. $x'\in \omega_\eps \setminus  \Delta_\eps$. In addition since $(\vec u_\eps)^{x'}_{\xi_\eps^{\pm}} =0$ $\LL^1$-a.e. in $(-2\sqrt 2 \eps,-\sqrt 2 \eps)$, it follows that $(\vec u_\eps)^{x'}_{\xi_\eps^{\pm}}(-\sqrt2 \eps)=0$. The statement concerning the vectors $\vec \eta^\pm_\eps$ can be proved in an analogous way.

\medskip

According again to slicing properties of functions of bounded deformations, we have that for $\LL^2$-a.e. $x' \in \omega$, the function
$x_3 \mapsto (u_\eps)_3(x',x_3)$ belongs to $SBV^2(0,1)$, and its jump set is contained in $\{x_3 \in (0,1) : (x',x_3) \in J_{\vec u_\eps}\} $. As a consequence, for $\LL^2$-a.e. $x \in \omega \setminus \Delta_\eps$, the function $x_3 \mapsto (u_\eps)_3(x',x_3)$ belongs to $H^1(0,1)$.
\end{proof}

The following technical result will be useful in the argument leading to a partial bulk energy lower bound.

\begin{lemma}\label{lem:app}
Let $\vartheta \in [0,2\pi)$, $p:= \cos \vartheta$, $q:=\sin \vartheta$, and define the unit vectors
$$\vec \xi^\pm:=\frac{1}{\sqrt 2}(\pm p,\pm q,1), \quad \vec \eta^\pm:=\frac{1}{\sqrt 2}(\mp q,\pm p,1).$$
For any matrix $\vec A=(a_{ij})_{1 \leq i,j \leq 3} \in \M^{3 \times 3}_{sym}$, we have the decomposition
\begin{eqnarray*}
|\vec A|^2 & =&|\vec A\vec \xi^+ \cdot\vec  \xi^+|^2+|\vec A\vec \xi^- \cdot \vec \xi^-|^2+|\vec A\vec \eta^+ \cdot\vec  \eta^+|^2+|\vec A\vec \eta^- \cdot \vec \eta^-|^2-\frac12 ({\rm tr} \vec A)^2\\
&&+ \frac12 |q^2 a_{11} + p^2 a_{22} -2pq a_{12}|^2+ \frac12 |p^2 a_{11} + q^2 a_{22} +2pq a_{12}|^2\\
&&+2|(p^2-q^2)a_{12} +pq(a_{22} - a_{11})|^2 + \frac12(a_{33}^2+(a_{11}+a_{22})^2).
\end{eqnarray*}
\end{lemma}

\begin{proof}
Let us define $\vec \xi_0:=\vec \xi^+ \wedge\vec  \xi^-=(q,-p,0)$ so that $\{\vec \xi^+,\vec \xi^-,\vec \xi_0\}$  is an orthonormal basis of $\R^3$. Then the family
$$\{\vec \xi^+ \otimes\vec  \xi^+, \vec \xi^- \otimes\vec  \xi^-, \vec \xi_0 \otimes \vec \xi_0, \sqrt 2 (\vec \xi^+ \odot \vec \xi_0), \sqrt 2(\vec \xi^- \odot \vec \xi_0), \sqrt 2(\vec \xi^+ \odot \vec \xi^-)\}$$
defines an orthonormal basis of the set $\M^{3 \times 3}_{sym}$, and  Pythagoras Theorem ensures that
\begin{align*}
|\vec A|^2 & = |\vec A : (\vec \xi^+ \otimes \vec \xi^+)|^2 +|\vec A:(\vec \xi^- \otimes\vec  \xi^-)|^2 + |\vec A:(\vec \xi_0 \otimes \vec \xi_0)|^2 \\
&\hspace{2cm}+  2 |\vec A:(\vec \xi^+ \odot \vec \xi_0)|^2+ 2|\vec A:(\vec \xi^- \odot \vec \xi_0)|^2+ 2|\vec A:(\vec \xi^+ \odot \vec \xi^-)|^2\\
&=|\vec A\vec \xi^+ \cdot \vec \xi^+|^2+|\vec A\vec \xi^- \cdot \vec \xi^-|^2+|\vec A\vec \xi_0 \cdot \vec \xi_0|^2  \nonumber \\
&\hspace{2cm}+  2 |\vec A\vec \xi^+ \cdot \vec \xi_0|^2+ 2|\vec A\vec \xi^- \cdot \vec \xi_0|^2+ 2|\vec A\vec \xi^+ \cdot \vec \xi^-|^2. 
\end{align*}
The conclusion follows from a straightfoward computation of each term.
\end{proof}

We now prove a partial bulk energy lower bound.

\begin{lemma}\label{lem:bulkb2}
Assume that $\lambda_b \geq \mu_b$. Then $(1- \theta)u_3=0$ $\LL^2$-a.e. in $\omega$, and
\begin{equation}\label{eq:bad2}
\liminf_{\eps \to 0}J_\eps(\vec u_\eps,\O_b) \geq \frac{\mu_b}{2} \liminf_{\eps \to 0} \int_{(\omega \setminus \Delta_\eps) \times (0,1)}Ê \left| u_\alpha(x) + \int_0^{x_3} \partial_\alpha  (u_\eps)_3(x',s)\dd s  \right|^2\dd x.
\end{equation}
If in addition the sequences $(\partial_\alpha (u_\eps)_3)_{\eps>0}$ are bounded in $L^2(\O_f)$, then
$$\liminf_{\eps \to 0}J_\eps(\vec u_\eps,\O_b) \geq \frac{\mu_b}{2}  \int_{\omega}(1-\theta)|\bar {\vec u}|^2\dd x'.$$
\end{lemma}

\begin{proof}
Let us denote by
$$\vec A_\eps := \matrizzz{
				\eps e_{11}(\vec u_\eps) & \eps e_{12}(\vec u_\eps) & e_{13}(\vec u_\eps) \\ 
				\eps e_{12}(\vec u_\eps) & \eps e_{22}(\vec u_\eps) & e_{23}(\vec u_\eps) \\
				e_{13}(\vec u_\eps) & e_{23}(\vec u_\eps) & \eps^{-1} e_{33}(\vec u_\eps)}.$$
the scaled strain so that
$$J_{\eps}(\vec u_\eps, \Omega_b) = \frac{\lambda_b}{2} \int_{\O_b} {\rm tr}(\vec A_\eps)^2 \dd x + \mu_b \int_{\Omega_b} \left | \vec A_\eps\right |^2 \dd x.$$

According to Lemma \ref{lem:app} with the angle $\vartheta=0$, we get that
\begin{multline}\label{eq:sum}
J_{\eps}(\vec u_\eps, \Omega_b)\geq \frac{\lambda_b-\mu_b}{2}\int_{\O_b}{\rm tr}(\vec A_\eps)^2\dd x\\
 + \mu_b\int_{\O_b} \left[|\vec A_\eps \xi^+ \cdot \xi^+|^2+|\vec A_\eps \xi^- \cdot \xi^-|^2+|\vec A_\eps \eta^+ \cdot \eta^+|^2+|\vec A_\eps \eta^- \cdot \eta^-|^2\right] \dd x\\
\geq  \mu_b\int_{\O_b} \left[|\vec A_\eps \xi^+ \cdot \xi^+|^2+|\vec A_\eps \xi^- \cdot \xi^-|^2+| \vec A_\eps \eta^+ \cdot \eta^+|^2+|\vec A_\eps \eta^- \cdot \eta^-|^2\right] \dd x,
\end{multline}
since $\lambda_b \geq \mu_b$. It remains to compute each of the four terms in the right hand side of the previous expression.  Let us start with the first term.
Changing variable $x=(y',0) + s\vec \xi_\eps^+$ (with $\dd x=(\sqrt 2 \eps)^{-1} \dd y'\dd s$), and using Fubini's Theorem, we get that
\begin{eqnarray*}
\int_{\O_b} |\vec A_\eps \vecg \xi^+ \cdot \vecg\xi^+|^2\dd x & \geq & \eps^2 \int_{(\omega_\eps \setminus \Delta_\eps) \times (-1,0)} |\nabla \vec u_\eps \vecg \xi^+_\eps \cdot \vecg\xi^+_\eps|^2\dd x\\
&\geq & \eps^2 \int_{\omega_\eps\setminus \Delta_\eps} \fint_{-\sqrt{2} \eps}^0 |\nabla \vec u_\eps((y',0)+s\vec \xi_\eps^+) \vecg \xi^+_\eps \cdot \vecg\xi^+_\eps|^2 \dd s \dd y'.
\end{eqnarray*}
According to Lemma \ref{lem:Sobolev}, since $(\vec u_\eps)^{ y'}_{\vecg\xi^+_\eps} \in H^1(-\sqrt 2 \eps,\sqrt 2 \eps)$ and $(\vec u_\eps)^{y'}_{\xi_\eps^{\pm}}(-\sqrt2 \eps)=0$ for $\LL^2$-a.e. $y' \in \omega \setminus \Delta_\eps$, we get that
\begin{eqnarray*}
\int_{\O_b} |\vec A_\eps \vecg \xi^+ \cdot \vecg\xi^+|^2\dd x &\geq & \eps^2 \int_{\omega_\eps\setminus  \Delta_\eps} \fint_{-\sqrt{2} \eps}^0 \left|\frac{d}{ds}[\vec u_\eps((y',0)+s\vec \xi_\eps^+) \cdot \vec \xi_\eps^+]\right|^2 \dd s \dd y'\\
& \geq & \eps^2 \int_{\omega_\eps\setminus \Delta_\eps}  \left| \fint_{-\sqrt{2} \eps}^0 \frac{d}{ds}[\vec u_\eps((y',0)+s\vec \xi_\eps^+) \cdot \vec \xi_\eps^+] \dd s \right|^2\dd y'\\
& = & \frac14 \int_{\omega_\eps\setminus  \Delta_\eps}\left |(u_\eps)^-_1 (y',0)+\frac{1}{\eps}(u_\eps)^-_3(y',0)\right|^2\dd y',
\end{eqnarray*}
where $\vec u_\eps^-(\cdot,0)$ denotes the lower trace of $\vec u_\eps$ on $\omega \times \{0\}$. Using again Lemma \ref{lem:Sobolev}, the function $(\vec u_\eps)^{y'}_{\vecg\xi^+\eps} \in H^1(-\sqrt 2 \eps,\sqrt 2 \eps)$ does not jump at $t=0$. Thus according to \cite[Theorem 4.5 (iv)]{AmCoDM97}, it follows that 
$$(u_\eps)^-_1 +\eps^{-1}(u_\eps)^-_3= (u_\eps)^+_1 +\eps^{-1}(u_\eps)^+_3 \quad \HH^2\text{-a.e. on }\omega \times \{0\},$$
and therefore,
\begin{equation}\label{eq:xi+}
\int_{\O_b} |\vec A_\eps \vecg \xi^+ \cdot \vecg\xi^+|^2\dd x\geq \frac14 \int_{\omega_\eps\setminus  \Delta_\eps} \left|(u_\eps)^+_1 (y',0)-\frac{1}{\eps}(u_\eps)^+_3(y',0)\right|^2\dd y'.
\end{equation}

Analogously, we can show that
\begin{align}
\int_{\O_b} |\vec A_\eps \vecg \xi^- \cdot \vecg\xi^-|^2\dd x&\geq \frac14 \int_{\omega_\eps\setminus  \Delta_\eps} \left|(u_\eps)^+_1 (y',0)-\frac{1}{\eps}(u_\eps)^+_3(y',0)\right|^2\dd y',\label{eq:xi-} \\
\int_{\O_b} |\vec A_\eps \vecg \eta^+ \cdot \vecg\eta^+|^2\dd x&\geq \frac14 \int_{\omega_\eps\setminus  \Delta_\eps} \left|(u_\eps)^+_2 (y',0)+\frac{1}{\eps}(u_\eps)^+_3(y',0)\right|^2\dd y', \label{eq:eta+}\\
\int_{\O_b} |\vec A_\eps \vecg \eta^- \cdot \vecg\eta^-|^2\dd x&\geq \frac14 \int_{\omega_\eps\setminus  \Delta_\eps} \left|(u_\eps)^+_2 (y',0)-\frac{1}{\eps}(u_\eps)^+_3(y',0)\right|^2\dd y'.\label{eq:eta-}
\end{align}

Summing up \eqref{eq:xi+}, \eqref{eq:xi-}, \eqref{eq:eta+}, \eqref{eq:eta-} and using \eqref{eq:sum} leads to
\begin{equation}\label{eq:1511}
J_{\eps}(\vec u_\eps, \Omega_b)  \geq  \frac{\mu_b}{2} \int_{\omega_\eps\setminus  \Delta_\eps} (u_\eps)^+_\alpha(y',0) (u_\eps)^+_\alpha(y',0)\dd y' + \frac{\mu_b}{\eps^2}\int_{\omega_\eps\setminus  \Delta_\eps}|(u_\eps)^+_3(y',0)|^2\dd y'.
\end{equation}
Since  $P_\eps \subset  \Delta_\eps$, Lemma \ref{lem:Sobolev} together with the fundamental Theorem of calculus yields,
\begin{equation}\label{eq:lowerbd2}
\int_{(\omega \setminus \Delta_\eps) \times (0,1)} \left|(u_\eps)_3(x',x_3) - (u_\eps)_3^+(x',0)\right|^2 \dd x 
\leq 4 \int_{\O_f} |e_{3 3}(\vec u_\eps)|^2 \dd x \leq C\eps^4.
\end{equation}
In particular, \eqref{eq:1511}, \eqref{eq:lowerbd2} and the energy bound \eqref{eq:nrjbound} ensure that
$$\int_{(\omega_\eps\setminus  \Delta_\eps) \times (0,1)} |(u_\eps)_3|^2\dd x \leq C\eps^2,$$
which implies, letting $\eps \to 0$, that $(1- \theta) u_3=0$ $\LL^2$-a.e. in $\omega$. Therefore Lemma \ref{lem:conv-impr} shows that $\chi_{\omega \setminus  \Delta_\eps} \partial_\alpha (u_\eps)_3  \weakcs 0$ weakly* in $L^2(\omega;H^{-1}(0,1))$. In addition, since $P_\eps \subset  \Delta_\eps$, we can use \eqref{eq:lowerbd1} and the fact that $(u_\eps)_\alpha \to u_\alpha$ strongly in $L^2(\O_f)$, to obtain \eqref{eq:bad2}.

\medskip

Assume now that the sequences $(\partial_\alpha (u_\eps)_3)_{\eps>0}$ are bounded in $L^2(\O_f)$. Then the convergence of the planar gradient improves to
 $\chi_{\omega \setminus  \Delta_\eps} \partial_\alpha (u_\eps)_3  \weakc 0$ weakly in $L^2(\O_f)$, and thus \eqref{eq:bad2} gives
$$\liminf_{\eps \to 0}J_\eps(\vec u_\eps,\O_b) \geq \frac{\mu_b}{2}  \int_{\O_f}(1-\theta)|\vec u|^2\dd x=\frac{\mu_b}{2} \int_{\omega}(1-\theta)|\bar{\vec u}|^2\dd x',$$
since $\theta$ is independent of $x_3$.
\end{proof}

\end{document}